\documentclass[reqno, english]{amsart}
\usepackage{amsfonts, amsmath, latexsym, verbatim, amscd, mathrsfs, color, array, titletoc}
\usepackage[colorlinks,
linkcolor=red,
anchorcolor=blue,
citecolor=blue]{hyperref}

\allowdisplaybreaks

\usepackage{bbm}
\usepackage{float}
\usepackage{amsmath, amssymb, amsthm, amsfonts, graphicx, color, mathtools}
\usepackage[hmargin=2.5cm, vmargin=2.5cm]{geometry}
\usepackage{bbm}
\usepackage{pdfsync}
\usepackage{epstopdf}
\usepackage{cite}
\usepackage{multirow}
\usepackage[font=bf,aboveskip=15pt]{caption}
\usepackage[toc,page]{appendix}

\usepackage{enumerate}
\usepackage{bookmark}
\DeclareFontShape{OT1}{cmr}{bx}{sc}{<-> cmbcsc10}{}
\usepackage{ulem}
\usepackage{syntonly}
\usepackage{upgreek}
\usepackage{bm}
\usepackage{nicematrix}
\usepackage{lipsum}

\allowdisplaybreaks[4]

\newcommand{\h}{\hat  }

\newcommand{\R} {\mathbb R}

\newcommand{\cuad}{{\sqcap\kern-.68em\sqcup}}

\newcommand{\be}{\begin{equation}}
	\newcommand{\ee}{\end{equation}}

\newtheorem{lemma}{Lemma}[section]
\newtheorem{prop}{Proposition}[section]
\newtheorem{theorem}{Theorem}[section]
\newtheorem{corollary}{Corollary}[section]
\newtheorem{remark}{Remark}[section]
\newcommand{\bremark}{\begin{remark} \em}
	\newcommand{\eremark}{\end{remark} }

\numberwithin{equation}{section}
\begin{document}
	\title[Rotating spirals for three-component competition systems]{Rotating spirals for three-component competition systems}
	
	\author[Z. Li]{Zaizheng Li}
	\address{\noindent
		School of Mathematical Sciences, Hebei Normal University, Shijiazhuang 050024, Hebei, P.R. China}
	\email{zaizhengli@hebtu.edu.cn}

	\author[S. Terracini]{Susanna Terracini}
	\address{\noindent
	Dipartimento di Matematica “Giuseppe Peano”, Universit\`a di Torino, Via Carlo Alberto, 10, 10123 Torino, Italy}
	\email{susanna.terracini@unito.it}


	\begin{abstract}
		We investigate the existence of rotating spirals for three-component competition-diffusion systems  in $B_1\subset \R^2$:
		\begin{equation*}%\label{system-t}
		\begin{cases}
			\partial_tu_1-\Delta u_1=f(u_1)-\beta  \alpha u_1u_2-\beta \gamma u_1 u_3,& \text{in}\ B_1\times \mathbb{R}^+,\\
			\partial_tu_2-\Delta u_2=f(u_2)-\beta \gamma u_1u_2-\beta \alpha u_2 u_3,& \text{in}\ B_1\times \mathbb{R}^+,\\
			\partial_tu_3-\Delta u_3=f(u_3)-\beta \alpha u_1u_3-\beta \gamma u_2 u_3,& \text{in}\ B_1\times \mathbb{R}^+,\\
			u_i(\textbf{x},0)=u_{i,0}(\textbf{x}), i=1,2,3, &\text{in} \ B_1,
		\end{cases}
	\end{equation*}
	with Neumann or Dirichlet boundary conditions,
	where $f(s)=\mu s(1-s)$, $\mu, \beta>0$, $\alpha>\gamma>0$.
		For the Neumann problem, we establish the existence of rotating spirals by applying the multi-parameter bifurcation theorem. As a byproduct, the instability of the constant positive solution is proved.
		In addition, for the non-homogeneous Dirichlet problem, the Rothe fixed point theorem is employed to prove the existence of rotating spirals. %This is the first existence result dealing with nonlinear $f$.
		\\
		\medskip
		
	\end{abstract}

	\maketitle
	
	%\tableofcontents

%	\bigskip

	\section{Introduction and main results}

This paper aims to contribute to the understanding of spatial patterning for diffusive Gause-Lotka-Volterra systems in the form
		\begin{equation}\label{systemLK}
		\partial_tu_i- d_i\Delta u_i=f_i(u_i)-\sum_{j\neq i}a_{ij} u_iu_j\;,
		\end{equation}				
in some domain $\Omega$, where $u_i(\textbf{x},t)$ represents the (non negative)  population density of the $i$-th species at time $t$ ($1\le i\le J$), 
$d_i\nabla u_i$ is the flow due to dispersal,  $f_i(u_i)$ accounts to the internal growth, and $A=(a_{ij})_{ij}$ is the interspecific interaction rates matrix. When the off-diagonal coefficients are all negative, the interaction is purely competitive, and two cases are distinguished according to whether the matrix $A$ is symmetric or not.  An asymmetric matrix accounts for a gap in the mutual aggressiveness between populations.  In ecology, these models help explain the spatial distribution of species and the formation of ecological patterns in diverse ecosystems (see, e.g., \cite{CanCosBook}).

We deal with the case of competing species in a homogeneous environment,  which corresponds to constant coefficients $d_i, a_{ij}$, space-independent reactions $f_i$, and no flux at the boundary (Neumann boundary conditions).  A natural question is whether such competing species could coexist in spatially homogeneous environments by forming non-constant distributions, or one or more of them is doomed to extinction as $t\to +\infty$.  More precisely, since the system admits constant solutions in space and time, the question is whether these are stable or not and whether there are others that do not lead to the asymptotic extinction of some of the populations. This is connected with the so-called Gause's law and Grinnell's competitive exclusion principle in ecology: \textit{Two species of approximately the same food habits are not likely to remain long evenly balanced in numbers in the same region. One will crowd out the other}. The perimeter within which this statement is true within systems of type \eqref{systemLK} has been the subject of intense study in the 1980s, starting from seminal papers  \cite{Kis81, Kis82, KisMimYos83, KisWei1985, Mim79, Mim1981}. Research in the field developed over the following decades to the recently proposed new predator-prey models with many species \cite{BerZil2018, BerZil2019}.

In this paper, we deal with a purely competitive three-population case,  when the three populations are interchangeable in terms of their diffusion speed and internal dynamics and are cyclically interchangeable in terms of intra-specific competition coefficients. Thus we are concerned with the following  reaction-diffusion system in the unit ball $B_1\subset \mathbb{R}^2$:
		\begin{equation}\label{system-t}
		\begin{cases}
			\partial_tu_1-\Delta u_1=f(u_1)-\beta  \alpha u_1u_2-\beta \gamma u_1 u_3,& \text{in}\ B_1\times \mathbb{R}^+,\\
			\partial_tu_2-\Delta u_2=f(u_2)-\beta \gamma u_1u_2-\beta \alpha u_2 u_3,& \text{in}\ B_1\times \mathbb{R}^+,\\
			\partial_tu_3-\Delta u_3=f(u_3)-\beta \alpha u_1u_3-\beta \gamma u_2 u_3,& \text{in}\ B_1\times \mathbb{R}^+\\
			u_i(\textbf{x},0)=u_{i,0}(\textbf{x}), i=1,2,3, &\text{in} \ B_1,
			\end{cases}
	\end{equation}
	with Neumann or Dirichlet boundary conditions,	  logistic reaction terms $f(s)=\mu s(1-s)$, $\mu>0$, $\beta>0$, $\alpha>\gamma>0$.
Note that the asymmetry affects cyclically the system, and the gap in the mutual aggressiveness between populations is quantified in the ratio $\alpha/\gamma$.   We are interested in coexistence solutions for possibly high values of the competition parameter $\beta$.

	The high competition regime for Lotka-Volterra multispecies systems (i.e., \eqref{systemLK} with $a_{ij}=\beta \bar a_{ij}$ and $\beta\to\infty)$ has been the focus of research over the last two decades, both in the stationary and evolutionary case. In seminal papers, for the stationary solutions of a class of systems including \eqref{system-t}, Dancer-Du \cite{Dancer-Du-1995, Dancer-Du-1995-2} obtained existence results for positive solutions to the homogeneous Dirichlet problems, while Conti-Terracini-Verzini \cite{T-2005} provided existence for the non-homogeneous boundary conditions. In addition, \cite{T-2005} fully describes the asymptotic behaviour of positive solutions as the competition rates go to infinity for the two-species case and prove some regularity properties for the multi-species case. Moreover,  the free boundary structure of the limit system was also established in \cite{Caffarelli, T-2006, T-2009, Spiraling-2019}. Bartsch-Tian-Wang \cite{Bartsch-2013, Bartsch-2015} obtained some bifurcation results from the synchronized solution branch for a multi-component variational system. Going back to the time-dependent system, Dancer-Wang-Zhang \cite{wang-zhang-2010, Dancer-2011, Dancer-2012}  studied the existence, uniqueness, uniform regularity, and free boundary for multi-species system. 
	
	When $\beta\to\infty$, stationary solutions admit segregated limiting profiles, giving rise to interesting spatial patterns, whose nodal set geometry was described in \cite{Caffarelli, Spiraling-2019}. More precisely, \cite{Spiraling-2019} linked the asymmetry of the interspecific competition rates $a_{ij}$ to the presence of nodal lines spiralling towards their junctions. This phenomenon raises the spontaneous question of the existence of rotating spirals, i.e., rigidly rotating movements for competitive systems in the asymmetrical case.
Indeed, Murakawa-Ninomiya \cite{MN-2011} firstly detected the rotating spirals when conducting numerical simulations on system \eqref{system-t}. Subsequently, Salort-Terracini-Verzini-Zilio \cite{susanna-2022} rigorously proved the existence of rotating spiralling solutions in the segregated limiting profiles and described the free boundary structure by a constructive method, but only in the case of linear reactions $f(s)=\mu s$. Sandstede-Scheel \cite{Spiral-waves2023} investigated the properties of rotating spirals for general reaction-diffusion equations (in unbounded domains) and provide some fundamental conceptual methods. 		The present work explores the nonlinear dynamics, providing a distinct perspective. In fact, we deal with the logistic reaction case $f(s)=\mu(1-s)$, and, contrary to \cite{susanna-2022},  we make use of a multi-parameter bifurcation theorem.
		
Rigidly rotating waves are stationary solutions in a uniformly rotating frame, that is
		\begin{equation}\label{rotate-ansatz}
			u_i(\textbf{x},t)=u_i(R_{\omega t}\textbf{x}),\quad i=1,2,3,\ 
		\text{where }\ 
			R_{\omega t}\textbf{x}=\begin{pmatrix}
				\cos (\omega t)& -\sin (\omega t)\\
				\sin (\omega t)& \cos(\omega t)
			\end{pmatrix}\begin{pmatrix}
				x_1\\x_2
			\end{pmatrix}.
		\end{equation}
		The $R_{\omega t}$ is called the rotation matrix of angular speed $\omega$. 
		After inserting \eqref{rotate-ansatz}
		into the system \eqref{system-t}, we obtain 
		\begin{equation}\label{sys-ome}
			\begin{cases}
				-\Delta u_1+\omega \textbf{x}^{\bot}\cdot \nabla u_1=\mu u_1(1-u_1)-\beta \alpha u_1u_2-\beta \gamma u_1 u_3,& \text{in}\ B_1,\\
				-\Delta u_2+\omega  \textbf{x}^{\bot}\cdot \nabla u_2=\mu u_2(1-u_2)-\beta \gamma u_1u_2-\beta \alpha u_2 u_3,& \text{in}\ B_1,\\
				-\Delta u_3+\omega  \textbf{x}^{\bot}\cdot \nabla u_3=\mu u_3(1-u_3)-\beta \alpha u_1u_3-\beta \gamma u_2 u_3,& \text{in}\ B_1,
			\end{cases}
		\end{equation}
with Neumann or Dirichlet boundary conditions.  
Since 	$\textbf{x}^{\bot}\cdot \nabla u_i=\frac{\partial u_i(r,\theta)}{\partial \theta}$ in polar coordinates, if $u_i$ is radial, then $\textbf{x}^{\bot}\cdot \nabla u_i=0$ and $\omega$ disappears.  
To obtain authentic rotating waves for the system \eqref{sys-ome}, our goal is to establish the existence of non-radial positive solutions. Therefore, it is essential to prove both the existence of positive solutions and their non-radial nature. If we succeed in finding a non-radial positive solution for \eqref{sys-ome}, we can assert the existence of rotating waves for the original system \eqref{system-t} through the relation \eqref{rotate-ansatz}.

To state our main results, we first introduce some notations that we will use
throughout the paper.
%\textbf{Notations.}
\begin{itemize}
	\item Define 	\begin{equation*}
		\begin{aligned}
			L_{\omega}u:=-\Delta u+\omega  \textbf{x}^{\bot}\cdot \nabla u=-\Delta u+\omega (-x_2\partial_1u+x_1\partial_2u).
		\end{aligned}
	\end{equation*}
	\item  $ -\Delta_N$ (resp. $ -\Delta_D$ )
	denotes the operator $-\Delta$ in a ball with the Neumann (resp. Dirichlet) boundary condition.
	\item $\lambda_{n}$ denotes a  positive eigenvalue of $-\Delta_N$, where $n=1, 2, \cdots$. 
	\item $f_{n,k}(r)$  is the solution of 
	\begin{equation*}
			f''(r)+\frac{1}{r}f'(r)+(\lambda_n -\frac{k^2}{r^2})f(r)=0, \quad f'(0)=f'(1)=0,\quad r\in(0,1).
	\end{equation*}
	\item  Recall that the spectrum $\sigma(T)$ of a linear operator $T$ consists of three parts: point spectrum $\sigma_p(T)$, continuous spectrum $\sigma_c(T)$, and residual spectrum $\sigma_r(T)$ (see \cite[Definition 7.2-1]{FA-Kreyszig}, \cite[p209]{FA-Yosida}). 
	The point spectrum $\sigma_p(T)$ is the set of eigenvalues. 
\end{itemize}
		Ikeno-Shigekawa \cite{IS-2009} described the spectrum of $-\Delta+\textbf{x}^{\bot}\cdot \nabla$ on the entire  space $\mathbb{R}^2$. To obtain the existence of rotating waves, 	we start by directly studying the spectrum and eigenspaces of  $L_{\omega}$ in the unit ball.   Indeed,  the spectrum and eigenspaces for $L_{\omega}$ are directly related to the ones of the Laplacian (with Neumann boundary conditions).
		\begin{prop}\label{spectrum-countable}
			For the operator $L_{\omega}$ with Neumann boundary condition,
			the spectrum $\sigma(L_{\omega})$ is a  countable  set.  More precisely, 
			\begin{equation*}
				\sigma(L_{\omega})= \sigma_{p}(L_{\omega})=\{0, \lambda_{n}+i\omega k: k\in \mathbb{Z}, \lambda_{n} \  \mbox{is\ the\  positive\ eigenvalue\ of\ } -\Delta_N \}.
			\end{equation*}
			Furthermore, \begin{itemize}
				\item[(i)] Every real eigenvalue is simple, i.e. $\dim\big(\cup_{m=1}^{\infty}\ker (L_{\omega}-\lambda)^m\big)=1$. Moreover, the corresponding eigenfunctions are radial.
				\item[(ii)] Every complex eigenvalue is double, i.e. $\ker(L_{\omega}-\lambda_{n}-\mathrm{i}k\omega)=\text{span}\{f_{n,k}(r)\mathrm{e}^{\mathrm{i}k\theta}, \mathrm{i}k f_{n,k}(r)\mathrm{e}^{\mathrm{i}k\theta}\}$.
				\item[(iii)] If $f_{n,k}(r)\mathrm{e}^{\mathrm{i}k\theta}$ is the eigenfunction of $L_{\omega}$ with respect to $\lambda_{n}+\mathrm{i}\omega k$, then  $f_{n,k}(r)\mathrm{e}^{-\mathrm{i}k\theta}$ is the eigenfunction of $L_{-\omega}$ with respect to ${\lambda_n}+\mathrm{i}\omega k$.
			\end{itemize}
		\end{prop}
		%
	%		\begin{definition}
		%
		%\end{definition}
		%
		We state the existence of rotating waves for the Neumann boundary condition below. 
		\begin{theorem}\label{Neumann}
			Consider the Neumann problem:
				\begin{equation*}%\label{system2}
				\begin{cases}
					L_{\omega}u_1=\mu u_1(1-u_1)-\beta \alpha u_1u_2-\beta \gamma u_1 u_3,& \text{in}\ B_1,\\
					L_{\omega}u_2=\mu u_2(1-u_2)-\beta \gamma u_1u_2-\beta \alpha u_2 u_3,& \text{in}\ B_1,\\
					L_{\omega}u_3=\mu u_3(1-u_3)-\beta \alpha u_1u_3-\beta \gamma u_2 u_3,& \text{in}\ B_1,\\
					\frac{\partial u_1}{\partial n}=	\frac{\partial u_2}{\partial n}=	\frac{\partial u_3}{\partial n}=0,&\text{on}\ \partial B_1,
				\end{cases}
			\end{equation*}
				where $\mu>0, \beta>0$ and  $\alpha>\gamma>0$. Then the following statements hold.
				\begin{itemize}
				\item[(i)] $\frac{\mu}{\mu+\beta\alpha+\beta\gamma}(1,1,1)$ is the unique positive constant solution. 
				\item[(ii)] Given fixed values for the parameters $\mu$, $\alpha$, and $\gamma$,  we assert that  $\big(\frac{\mu}{\mu+\beta\alpha+\beta\gamma}(1,1,1),  (\beta,\mu)\big)$ is a bifurcation point for any positive  $(\beta, \mu)$ satisfying
				\begin{equation}\label{beta-star}
			\mu=\sqrt{3}k\omega\frac{\alpha+\gamma}{\alpha-\gamma}-\lambda_n,\quad		\beta=\frac{1}{\alpha+\gamma}\frac{2\mu(\mu+\lambda_n)}{\mu-2\lambda_n},% \quad n=1, 2, \cdots. 	
				\end{equation}
				 with $k\in\{3m-1: m\in\mathbb{N}\}$. Moreover, there exist positive rotating waves in the solution curve.  %More precisely, the rotating wave curve is of the form 
				\item[(iii)] For  fixed values of  $\mu,\alpha,\gamma$,  it holds that $\big(\frac{\mu}{\mu+\beta\alpha+\beta\gamma}(1,1,1), (\beta,\omega)\big)$ is a bifurcation point for  all $(\beta,\omega)$ satisfying 
				\begin{equation}\label{beta-star2}
					\omega=	\frac{\sqrt{3}}{3k}\frac{\alpha-\gamma}{\alpha+\gamma}(\mu+\lambda_n),\quad	\beta=\frac{1}{\alpha+\gamma}\frac{2\mu(\mu+\lambda_n)}{\mu-2\lambda_n}, % \quad n=1, 2, \cdots. 	
				\end{equation}
				 with $k\in\{3m-1: m\in\mathbb{N}\}$. Furthermore,  positive rotating waves exist in the solution curve.
			% More precisely, the rotating wave curve is of the form 
				 \item[(iv)] The solution curve  in $(ii)$ (resp. $(iii)$) is of the form 
			$ \textbf{u}(s)=\frac{\mu(s)}{\mu(s)+\beta(s)\alpha+\beta(s)\gamma}(1,1,1)+s\textbf{h}_0+\psi\big(s\textbf{h}_0, (\beta(s),\mu(s)\big), $ (resp. $ \textbf{u}(s)=\frac{\mu}{\mu+\beta(s)\alpha+\beta(s)\gamma}(1,1,1)+s\textbf{h}_0+\psi\big(s\textbf{h}_0, (\beta(s),\omega(s)\big)$) with 
			$s\in(-\delta,\delta)$ for some small $\delta>0$. And the direction of bifurcation is 	  \begin{equation*}
			  % \Big[\textbf{u}(s)-\frac{\mu(s)}{\mu(s)+\beta(s)\alpha+\beta(s)\gamma}(1,1,1)\Big]'(0)=
			   \textbf{h}_0=f_{n,k}(r)\begin{pmatrix}
		-\cos k\theta+\sqrt{3}\sin k\theta\\
		-\cos k\theta-\sqrt{3}\sin k\theta\\
		2\cos k\theta
		\end{pmatrix},
				 \end{equation*}
				 which is a non-radial, non-trivial function. Therefore, we obtain authentic rotating waves.
				 \item[(v)] The solution $\frac{\mu}{\mu+\beta\alpha+\beta\gamma}(1,1,1)$ is unstable if $\beta>\frac{2\mu}{\alpha+\gamma}$.
				\end{itemize} 
In addition, 		the positive rotating waves satisfy the property	$u_2(\textbf{x})=u_1(R_{\frac{2\pi}{3}}\textbf{x}), \  u_3(\textbf{x})=u_2(R_{\frac{2\pi}{3}}\textbf{x}).$
		\end{theorem}
		\begin{remark}
			$(i)$ Both of  \eqref{beta-star} and \eqref{beta-star2} imply that 
			$\beta>\frac{2\mu}{\alpha+\gamma}$ and $\mu>2\lambda_{n}$. 
		%	$(ii)$ The condition  implies 	$\beta>\frac{2\mu}{\alpha+\gamma}$.
		$	(ii)$
			If $\mu$ is close to $2\lambda_n$, the $\beta$ in \eqref{beta-star2} could be arbitrarily large.
			$(iii)$ By \cite[Remark I.19.7]{bifurcation-2012}, 
			in truth, we obtain a $2$-dimensional surface of nontrivial rigidly rotating waves. 
			(iv) By continuity, close to the bifurcation point, such waves are unstable. And by the analysis in \cite{SSW-1997, SSW-1997-2}, for any $(\beta,\mu)$ sufficiently close to the parameter satisfying \eqref{beta-star}, there exists an $SO(2,\mathbb{R})$-invariant, locally flow-invariant manifold $M^{cu}_{\beta,\mu}$ contained in $C(\overline{B}_{1})$, where $SO(2,\mathbb{R})$ is the group formed by rotation matrices. The manifold $M^{cu}_{\beta,\mu}$ and the action of $SO(2,\mathbb{R})$ on $M^{cu}_{\beta,\mu}$ are smooth and depend smoothly on the parameter $(\beta,\mu)$. Furthermore, $M^{cu}_{\beta,\mu}$ contains all solutions that stay close to the group orbit of the rotating wave for all negative times. Finally, $M^{cu}_{\beta,\mu}$ is locally exponentially attracting.
		\end{remark}

	The subsequent theorem demonstrates the existence of rotating waves for the non-homogeneous Dirichlet problems.
\begin{theorem}\label{Dirichlet}
		There exists at least a positive solution for the non-homogeneous Dirichlet problem:
	\begin{equation*}%\label{system2-dp}
		\begin{cases}
			L_{\omega}u_1=\mu u_1(1-u_1)-\beta \alpha u_1u_2-\beta \gamma u_1 u_3,& \text{in}\ B_1,\\
			L_{\omega}u_2=\mu u_2(1-u_2)-\beta \gamma u_1u_2-\beta \alpha u_2 u_3,& \text{in}\ B_1,\\
			L_{\omega}u_3=\mu u_3(1-u_3)-\beta \alpha u_1u_3-\beta \gamma u_2 u_3,& \text{in}\ B_1,\\
			(u_1,u_2,u_3)=(\phi_{1},\phi_{2},\phi_{3}),&\text{on}\ \partial B_1,
		\end{cases}
	\end{equation*}
	where $\mu,  \beta>0$, $\alpha>\gamma>0$,  and 
	\begin{equation*}
	\phi_{i}\in C(\partial B_{1}), 0\le \phi_{i}(\textbf{x})< 1, \phi_{i}\not\equiv 0,  \phi_{2}(\textbf{x})=\phi_{1}(R_{\frac{2\pi}{3}}\textbf{x}), \ \phi_{3}(\textbf{x})=\phi_{2}(R_{\frac{2\pi}{3}}\textbf{x}), \ i=1, 2, 3.
	\end{equation*}
	 Furthermore, 		the positive rotating wave satisfies	$u_2(\textbf{x})=u_1(R_{\frac{2\pi}{3}}\textbf{x}), \  u_3(\textbf{x})=u_2(R_{\frac{2\pi}{3}}\textbf{x}).$
\end{theorem}

Our approach is outlined as follows. At first, we need to characterize the precise spectrum of the operator		$L_{\omega}=-\Delta +\omega \textbf{x}^{\bot}\cdot \nabla$.  As this operator is not self-adjoint, the spectrum   $\sigma(L_{\omega})$ forms a closed set that includes complex numbers.
Moving on to the Neumann problem, we establish the existence of rotating spirals through the applications of multi-parameter bifurcation theorems. This not only confirms the presence of rotating spirals but also serves as an illustrative example of applying multi-parameter bifurcation theorems to partial differential systems. In the case of the non-homogeneous Dirichlet problem, we employ the Rothe fixed point theorem to prove the existence of rotating spirals.  This part of our study is particularly noteworthy as it marks the first exploration of rotating spirals within competition-diffusion systems featuring a nonlinear $f$.

The organization of this paper unfolds as follows. In Section 2, we present the spectrum of $L_{\omega}$ and demonstrate the existence result for the Neumann problem, including the proof of  Proposition \ref{spectrum-countable} and Theorem \ref{Neumann}. Section 3 illustrates the existence result for the non-homogeneous Dirichlet problem, along with the proof for Theorem \ref{Dirichlet}.

	\section{The NBC system}

	We consider the nonlinear elliptic system
	\begin{equation}\label{system2}
		\begin{cases}
			L_{\omega}u_1=\mu u_1(1-u_1)-\beta \alpha u_1u_2-\beta \gamma u_1 u_3,& \text{in}\ B_1,\\
			L_{\omega}u_2=\mu u_2(1-u_2)-\beta \gamma u_1u_2-\beta \alpha u_2 u_3,& \text{in}\ B_1,\\
			L_{\omega}u_3=\mu u_3(1-u_3)-\beta \alpha u_1u_3-\beta \gamma u_2 u_3,& \text{in}\ B_1,\\
			\frac{\partial u_1}{\partial n}=	\frac{\partial u_2}{\partial n}=	\frac{\partial u_3}{\partial n}=0,&\text{on}\ \partial B_1.
		\end{cases}
	\end{equation}
	We want to prove the existence of nontrivial positive ($	u_i>0,\  \forall i=1,2,3$) symmetric solutions such that
	\begin{equation*}
		u_2(\textbf{x})=u_1(R_{\frac{2\pi}{3}}\textbf{x}), \  u_3(\textbf{x})=u_2(R_{\frac{2\pi}{3}}\textbf{x}).
	\end{equation*}
	
	\begin{proof}[Proof of Theorem \ref{Neumann}(i)]
			It is easy to observe that 
		\begin{equation*}
			(u_1,u_2, u_3)=\frac{\mu}{\mu+\beta \alpha+\beta \gamma}(1,1,1)
		\end{equation*}
		is the unique constant positive solution to \eqref{system2} by the Cramer's rule and Proposition \ref{det-positive}.
	\end{proof}

 \begin{lemma}%\label{upperbound-ui}
	Any positive solution of \eqref{system2} is  such that $0<u_i<1$ in $\overline{B_1}$ for any $i=1, 2, 3$.
\end{lemma}
\begin{proof}
	Take  $w_i=u_{i}-1$, then 
	\begin{equation*}
		\begin{cases}
			(L_{\omega}+\mu)w_i <0,&\text{in}\ B_1,\\
			\frac{\partial w_i}{\partial n}=0,&\text{on} \ \partial B_1.
		\end{cases}
	\end{equation*}
	By the \cite[ Maximum Principle 4.1]{amann-1976},  we obtain $w_i<0$ in $\overline{B_1}$.
	Finally, we conclude that $0<u_i<1$  in $\overline{B_1}$.
\end{proof}
	
\medskip
\subsection{The spectrum of $L_{\omega}$ }
For the linear operator $L_{\omega}$
with the Neumann boundary condition, we analyze the spectrum set. 

\begin{lemma}(\cite[Appendix A, p29]{guo2023-JMPA}, \cite{W-1996})\label{polar-represent}
	For function $\psi(x)$ defined on a ball in $\mathbb{R}^{2}$, consider the polar coordinate $(r, \theta)$ in $\mathbb{R}^{2}$, where $\theta\in[0,2\pi]$. Then
	\begin{equation*}
		\nabla \psi=\frac{\textbf{x}}{r}\psi_{r}+\frac{\textbf{x}^{\bot}}{r^{2}}\psi_{\theta},\qquad
		\textbf{x}^{\bot}\cdot \nabla \psi=\frac{\partial \psi(r,\theta)}{\partial \theta},
	\end{equation*} 
	where $\textbf{x}^{\bot}:=(-x_{2},x_{1})$ with $\textbf{x}=(x_{1},x_{2})\in \mathbb{R}^{2}$.
\end{lemma}
\begin{proof}[Proof of Proposition \ref{spectrum-countable}]
	On the one hand, $u=constant$ is an eigenfunction with respect to eigenvalue $\lambda=0$ for the Neumann boundary condition.
	On the other hand, 
	By Lemma \ref{polar-represent}, in polar coordinates, $$	L_{\omega}u=-u_{rr}-\frac{1}{r}u_r-\frac{1}{r^2}u_{\theta\theta}+\omega u_{\theta}.$$
	If we try to find eigenfunctions of the form $f(r)\mathrm{e}^{\mathrm{i}k\theta}$, where $k\in\mathbb{Z}$. Then 
	\begin{equation}\label{bessel}
		\begin{aligned}
			f''(r)+\frac{1}{r}f'(r)+(\lambda-\mathrm{i}\omega k -\frac{k^2}{r^2})f(r)&=0, 
		\end{aligned}
	\end{equation}
	Recall that the Bessel function is used as the eigenfunctions in \cite[Proposition 1.2.14]{eigenvalue2006} or   \cite[p302]{courant-book}, both the Dirichlet boundary and Neumann boundary included. 	If $\lambda_{n,k}=\lambda_{n}+\mathrm{i}\omega k$, then the solution  $f_{n,k}(r)$  of equation \eqref{bessel}  exists,  and $f_{n,k}(r)\mathrm{e}^{\mathrm{i}k\theta}$ is the eigenfunction corresponding to $\lambda_{n,k}=\lambda_{n}+\mathrm{i}\omega k$.

	Take $\lambda\neq \lambda_{n}+\mathrm{i}\omega k$, We want to prove that for  $\forall f\in L^{2}(B_{1})$, the equation 
	\begin{equation*}
		(L_{\omega}-\lambda)u=f
	\end{equation*}
	admits a unique solution. Indeed,
	by the Fourier series expansion ( see \cite[p135]{stein}), 
	\begin{equation*}
		f(r\mathrm{e}^{\mathrm{i}\theta})=\sum_{k=-\infty}^{\infty}f_k(r)\mathrm{e}^{\mathrm{i}k\theta}.
	\end{equation*}
	We try to find $u_k(r)\mathrm{e}^{\mathrm{i}k\theta}$ such that
	\begin{equation*}
		(L_{\omega}-\lambda)u_k(r)\mathrm{e}^{\mathrm{i}k\theta}=f_k(r)\mathrm{e}^{\mathrm{i}k\theta}.
	\end{equation*}
	That is,
	\begin{equation*}
		-u''_k(r)-\frac{1}{r}u'_k(r)+\big(\frac{k^2}{r^2}+\mathrm{i}\omega k-\lambda\big)u_k(r)=f_k(r),
	\end{equation*}
	which is equivalent to 
	\begin{equation*}
		(-\Delta +\mathrm{i}\omega k-\lambda)\left(u_k(r)\mathrm{e}^{\mathrm{i}k\theta}\right)=f_k(r)\mathrm{e}^{\mathrm{i}k\theta}.
	\end{equation*}
	Then $u_k$ exists and $\|u_k\|_{L^2}\le C\|f_k\|_{L^2}$ since $\lambda-\mathrm{i}\omega k\notin \sigma(-\Delta)$. The question is to find a constant $C=C(\text{Re}\lambda)$ independent of $k$, which is true by \cite[Theorem 5.8]{resolvent-book}.
\end{proof}

From the arguments above, we could also obtain the spectrum of $L_{\omega}$ with the Dirichlet boundary condition.
	\begin{corollary}%\label{spectrum-countable-d}
	For the operator $L_{\omega}$ with Dirichlet boundary condition,
	the spectrum $\sigma(L_{\omega})$ is a countable set.  More precisely, 
	\begin{equation*}
		\sigma(L_{\omega})= \sigma_{p}(L_{\omega})=\{ \lambda_{n}+i\omega k: k\in \mathbb{Z}, \lambda_{n} \  \mbox{is\ the\ eigenvalue\ of\ } -\Delta_D \}.
	\end{equation*}
\end{corollary}

\medskip
\subsection{Proof of Theorem \ref{Neumann}(ii)}
%\subsubsection{Bifurcation-$(\beta,\mu)$}
Our objective is to apply the multi-parameter bifurcation theorem, as outlined in \cite[Theorem I.19.6]{bifurcation-2012} or \cite[Theorem 1.3.4]{chang2005}, with fixed values for the parameters $\omega$, $\alpha$, and $\gamma$. 
%We  intend to apply  the multi-parameter  bifurcation theorem  \cite[Theorem I.19.6]{bifurcation-2012}  or \cite[Theorem 1.3.4]{chang2005} for  fixed $\omega,\alpha,\gamma$.
Throughout this subsection, we denote
\begin{equation*}
	E:=\{(u_1, u_2, u_3): u_{i}\in C(\overline{B_{1}},\mathbb{R}), u_2(\textbf{x})=u_1(R_{\frac{2\pi}{3}}\textbf{x}), u_3(\textbf{x})=u_2(R_{\frac{2\pi}{3}}\textbf{x}) )\}.
\end{equation*}
To locate non-constant solutions, consider the mapping $F: \big((C^{2,\alpha}(\overline{B_1},\mathbb{R}))^3\cap E\big)\times \mathbb{R}^2\rightarrow (C^{0,\alpha}(\overline{B_1},\mathbb{R}))^3\cap E$:
\begin{equation*}
	F\big((u_1,u_2,u_3),(\beta,\mu)\big)=\begin{pmatrix}
		L_{\omega}(u_1)-\mu u_1(1-u_1)+\beta \alpha u_1 u_2+\beta \gamma u_1 u_3,\\
		L_{\omega}(u_2)-\mu u_2(1-u_2)+\beta \gamma u_1 u_2+\beta \alpha u_2 u_3\\
		L_{\omega}(u_3)-\mu u_3(1-u_3)+\beta \alpha u_1 u_3+\beta \gamma u_2 u_3
	\end{pmatrix}.
\end{equation*}
Obviously, 
\begin{equation*}
	F((1,1,1),(0,\mu))=(0,0,0),\quad F\big( t_{\beta,\mu}\big(1,1,1\big),(\beta,\mu)\big)=(0,0,0),\ \text{where}\ t_{\beta,\mu}=\frac{\mu}{\mu+\beta(\alpha+\gamma)}.
\end{equation*}
%We hope 
To find a non-constant positive solution, we define the mapping $\tilde{F}:  \big((C^{2,\alpha}(\overline{B_1},\mathbb{R}))^3\cap E\big)\times \mathbb{R}^2\rightarrow (C^{0,\alpha}(\overline{B_1},\mathbb{R}))^3\cap E$: 
\begin{equation*}
	\begin{aligned}
		\tilde{F}((u_1,u_2,u_3),(\beta,\mu))&:={F}\big((u_1+t_{\beta,\mu},u_2+t_{\beta,\mu},u_3+t_{\beta,\mu}),(\beta,\mu)\big)
	=\begin{pmatrix}
			L_{\omega}(u_1)+(u_1+t_{\beta,\mu})(\mu u_1+\beta \alpha u_2+\beta \gamma u_3)\\
			L_{\omega}(u_2)+(u_2+t_{\beta,\mu})(\mu u_2+\beta \gamma u_1+\beta \alpha u_3)\\
			L_{\omega}(u_3)+(u_3+t_{\beta,\mu})(\mu u_3+\beta \alpha u_1+\beta \gamma u_2)
		\end{pmatrix}
		\\
		&=:L(\beta,\mu)\big(u_{1},u_{2},u_{3}\big)+P\big((u_{1},u_{2},u_{3}),(\beta,\mu)\big),
	\end{aligned}
\end{equation*}
where  
\begin{equation*}
	L(\beta,\mu)\big(u_{1},u_{2},u_{3}\big)=\begin{pmatrix}
		L_{\omega}(u_1)+t_{\beta,\mu}(\mu u_1+\beta \alpha u_2+\beta \gamma u_3)\\
		L_{\omega}(u_2)+t_{\beta,\mu}(\mu u_2+\beta \gamma u_1+\beta \alpha u_3)\\
		L_{\omega}(u_3)+t_{\beta,\mu}(\mu u_3+\beta \alpha u_1+\beta \gamma u_2)
	\end{pmatrix}, P\big((u_{1},u_{2},u_{3}),(\beta,\mu)\big)
	=\begin{pmatrix}
		u_1(\mu u_1+\beta \alpha u_2+\beta \gamma u_3)\\
		u_2(\mu u_2+\beta \gamma u_1+\beta \alpha u_3)\\
		u_3(\mu u_3+\beta \alpha u_1+\beta \gamma u_2)
	\end{pmatrix}.
\end{equation*}
Then 
$ \tilde{F}\big((0,0,0),(\beta,\mu)\big)=0$, $P\big((0,0,0),(\beta,\mu)\big)=0$ for all $\beta, \mu$. Furthermore, 
\begin{equation*}
	\begin{aligned}
		P_x\big((u_1,u_2,u_3),(\beta,\mu)\big)(h_1,h_2,h_3)
		&
		=\begin{pmatrix}
			u_1(\mu h_1+\beta \alpha h_2+\beta \gamma h_3)+h_1(\mu u_1+\beta \alpha u_2+\beta \gamma u_3)\\
			u_2(\mu h_2+\beta \gamma h_1+\beta \alpha h_3)+h_2(\mu u_2+\beta \gamma u_1+\beta \alpha u_3)\\
			u_3(\mu h_3+\beta \alpha h_1+\beta \gamma h_2)+h_3(\mu u_3+\beta \alpha u_1+\beta \gamma u_2)
		\end{pmatrix},
	\end{aligned}
\end{equation*}
\begin{equation*}
	\begin{aligned}
		P_{x\beta}\big((u_1,u_2,u_3),(\beta,\mu)\big)(h_1,h_2,h_3)
		&
		=\begin{pmatrix}
			u_1(\alpha h_2+ \gamma h_3)+h_1( \alpha u_2+\gamma u_3)\\
			u_2( \gamma h_1+\alpha h_3)+h_2( \gamma u_1+\alpha u_3)\\
			u_3( \alpha h_1+ \gamma h_2)+h_3( \alpha u_1+ \gamma u_2)
		\end{pmatrix},
	\end{aligned}
\end{equation*}
and
\begin{equation*}
	\begin{aligned}
		P_{x\mu}\big((u_1,u_2,u_3),(\beta,\mu)\big)(h_1,h_2,h_3)
		&
		=\begin{pmatrix}
			2u_1h_1\\
			2u_2h_2\\
			2u_3h_3
		\end{pmatrix}.
	\end{aligned}
\end{equation*}
Thus $P_x\big((0,0,0),(\beta,\mu)\big)=0$,
	$P_{x\beta}\big((0,0,0),(\beta,\mu)\big)=0$, and  $P_{x\mu}\big((0,0,0),(\beta,\mu)\big)=0$ for all $\beta, \mu$.

\textbf{Step 1:} We try to find the $\ker L(\beta,\mu)$ and $\text{Ran} L(\beta,\mu)$. Recall that
\begin{equation*}
	\begin{aligned}
		L(\beta,\mu)\big(h_{1},h_{2},h_{3}\big)	&
		= \begin{pmatrix}
			L_{\omega}(h_1)\\
			L_{\omega}(h_2)
			\\	L_{\omega}(h_3)
		\end{pmatrix}+	A_{\beta,\mu}
		\begin{pmatrix}
			h_1\\
			h_2\\
			h_3
		\end{pmatrix}, \text{where} \  A_{\beta,\mu}:=t_{\beta,\mu} \begin{pmatrix}
			\mu&\beta\alpha&\beta\gamma\\
			\beta\gamma&\mu&\beta\alpha\\
			\beta\alpha&\beta\gamma&\mu
		\end{pmatrix},
	\end{aligned}
\end{equation*}
For $A_{\beta,\mu}$, there are three eigenvalues:
\begin{equation*}
	\mu,\  \delta_{\beta,\mu}=\frac{2\mu-\beta(\alpha+\gamma)}{2}t_{\beta,\mu} +\mathrm{i}\frac{\sqrt{3}}{2}\beta(\alpha-\gamma)t_{\beta,\mu}, 
	\  \bar{\delta}_{\beta,\mu}=\frac{2\mu-\beta(\alpha+\gamma)}{2}t_{\beta,\mu} -\mathrm{i}\frac{\sqrt{3}}{2}\beta(\alpha-\gamma)t_{\beta,\mu}.
\end{equation*}
 Moreover,  by Proposition \ref{S}, one has
\begin{equation*}
	A_{\beta,\mu}S=S\begin{pmatrix}
		\mu&0&0\\
		0&\delta_{\beta,\mu}&0\\
		0&0&\bar{\delta}_{\beta,\mu}
	\end{pmatrix}, \quad S=	\begin{pmatrix}
		1&-\frac{1}{2}+\frac{\sqrt{3}}{2}\mathrm{i}&-\frac{1}{2}-\frac{\sqrt{3}}{2}\mathrm{i}\\
		1&-\frac{1}{2}-\frac{\sqrt{3}}{2}\mathrm{i}&-\frac{1}{2}+\frac{\sqrt{3}}{2}\mathrm{i}\\
		1&1&1
	\end{pmatrix}.
\end{equation*}
Consequently,
\begin{equation*}
	\begin{cases}
		L_{\omega} v_1+\mu v_1=0,\\
		L_{\omega} v_2+\delta_{\beta,\mu} v_2=0,\\
		L_{\omega} v_3+\bar{\delta}_{\beta,\mu} v_3=0,
	\end{cases} \text{where}	\begin{pmatrix}
		v_1\\v_2\\v_3
	\end{pmatrix}=S^{-1}\begin{pmatrix}
		h_1\\h_2\\h_3
	\end{pmatrix}.
\end{equation*}
By Proposition \ref{spectrum-countable}, it follows that   $\frac{\beta(\alpha+\gamma)-2\mu}{2}t_{\beta,\mu}$ is a positive  eigenvalue of $-\Delta_N$, and $\frac{\sqrt{3}}{2}\beta(\alpha-\gamma)t_{\beta,\mu}=k\omega$ with $k\in \mathbb{Z}$, that is,
\begin{itemize}
	\item[(i)]  $\frac{\beta(\alpha+\gamma)-2\mu}{2}\frac{\mu}{\mu+\beta(\alpha+\gamma)}\in \{\lambda_{n}:  \lambda_{n} \text{ \ is\  a\  positive\  eigenvalue\  of\ } -\Delta_N, n=1,2,\cdots \}$;
	\item [(ii)] $\frac{\sqrt{3}}{2}\beta(\alpha-\gamma)\frac{\mu}{\mu+\beta(\alpha+\gamma)}\in \{k\omega: k\in\mathbb{Z}\}$. 
\end{itemize}
%By straightforward calculations, it holds that
%\begin{equation}\label{beta-star}
%	\beta=\beta(\mu)=\frac{1}{\alpha+\gamma}\frac{2\mu(\mu+\lambda_n)}{\mu-2\lambda_n},\quad 	\mu=\sqrt{3}k\omega\frac{\alpha+\gamma}{\alpha-\gamma}-\lambda_n.
%\end{equation}
%Consequently, the second equation gives
Hence for  $\beta_{0}, \mu_0>0$ satisfying   \eqref{beta-star} with $\lambda_{n}$ and $k$  chosen properly and fixed,  
\begin{equation*}
	\ker L(\beta_{0},\mu_0)=\{(h_{1},h_{2},h_{3})\},\ \text{where}\ 
	\begin{pmatrix}
		h_1\\h_2\\h_3
	\end{pmatrix}=S\begin{pmatrix}
		0\\ (a-\mathrm{i}bk) f_{n,k}(r)\mathrm{e}^{-\mathrm{i}k\theta}\\(c+\mathrm{i}dk)f_{n,k}(r)\mathrm{e}^{\mathrm{i}k\theta}
	\end{pmatrix},\ \text{where} \ a,b, c,d \in\mathbb{R}.
\end{equation*}
To guarantee $h_{i}$ are real functions, we need $a=c, b=d$. Thus 
\begin{small}
\begin{equation}\label{kernel-expression}
	\begin{aligned}
		\begin{pmatrix}
			h_1\\h_2\\h_3
		\end{pmatrix}&=S\begin{pmatrix}
			0\\ (a-\mathrm{i}bk) f_{n,k}(r)\mathrm{e}^{-\mathrm{i}k\theta}\\(a+\mathrm{i}bk)f_{n,k}(r)\mathrm{e}^{\mathrm{i}k\theta}
		\end{pmatrix}=af_{n,k}(r)\begin{pmatrix}
			-\cos k\theta+\sqrt{3}\sin k\theta\\
			-\cos k\theta-\sqrt{3}\sin k\theta\\
			2\cos k\theta
		\end{pmatrix}+bkf_{n,k}(r)\begin{pmatrix}
			\sin k\theta+\sqrt{3}\cos k\theta\\
			\sin k\theta-\sqrt{3}\cos k\theta\\
			-2\sin k\theta
		\end{pmatrix}, \ a,b \in\mathbb{R}.
	\end{aligned}
\end{equation}
\end{small}
In order that $(h_{1},h_{2},h_{3})\in E$, it is required  that \textcolor{black}{$\frac{2k+2}{3}\in\{2n: n\in\mathbb{Z}\}$.} 
In this case, 
$\dim \ker L(\beta_{0},\mu_0)= 2$. 
On the other hand,  note that
\begin{equation*}
	[\text{Ran} L(\beta_{0},\mu_0)]^{\bot}= \ker \big(L(\beta_{0},\mu_0)\big)^*.
\end{equation*}
In what follows, we try to find $ \ker \big(L(\beta_{0},\mu_0)\big)^*$. Notice easily  that
\begin{equation*}
	\begin{aligned}
		(L(\beta,\mu))^{*}(v_{1},v_{2},v_{3})
		= \begin{pmatrix}
			L_{-\omega}(v_1)\\
			L_{-\omega}(v_2)
			\\	L_{-\omega}(v_3)
		\end{pmatrix}+	M_{\beta,\mu}
		\begin{pmatrix}
			v_1\\
			v_2\\
			v_3
		\end{pmatrix}, \text{where} \  M_{\beta,\mu}:=t_{\beta,\mu} \begin{pmatrix}
			\mu&\beta\gamma&\beta\alpha\\
			\beta\alpha&\mu&\beta\gamma\\
			\beta\gamma&\beta\alpha&\mu
		\end{pmatrix}.
	\end{aligned}
\end{equation*}
Hence by Proposition \ref{S} again, we obtain
\begin{equation*}
	M_{\beta,\mu}S=S\begin{pmatrix}
		\mu&0&0\\
		0&\bar{\delta}_{\beta,\mu}&0\\
		0&0&{\delta}_{\beta,\mu}
	\end{pmatrix}, \quad S=	\begin{pmatrix}
		1&-\frac{1}{2}+\frac{\sqrt{3}}{2}\mathrm{i}&-\frac{1}{2}-\frac{\sqrt{3}}{2}\mathrm{i}\\
		1&-\frac{1}{2}-\frac{\sqrt{3}}{2}\mathrm{i}&-\frac{1}{2}+\frac{\sqrt{3}}{2}\mathrm{i}\\
		1&1&1
	\end{pmatrix}.
\end{equation*}
Then  for $(v_{1},v_{2},v_{3})\in \ker \big(L(\beta_{0},\mu_0)\big)^*$,
\begin{equation*}
	\begin{cases}
		L_{-\omega} w_1+\mu w_1=0,\\
		L_{-\omega} w_2+(-\lambda_{n}-\mathrm{i}\omega k) w_2=0,\\
		L_{-\omega} w_3+(-\lambda_{n}+\mathrm{i}\omega k) w_3=0,
	\end{cases} \text{where}	\begin{pmatrix}
		w_1\\w_2\\w_3
	\end{pmatrix}=S^{-1}\begin{pmatrix}
		v_1\\v_2\\v_3
	\end{pmatrix}.
\end{equation*}
Thus by Proposition \ref{spectrum-countable}(iii),
\begin{equation*}
		\ker \big(L(\beta_{0},\mu_0)\big)^*=\Big\{(v_{1},v_{2},v_{3}),\ \text{where}\ 
	\begin{aligned}
		\begin{pmatrix}
			v_1\\v_2\\v_3
		\end{pmatrix}=S\begin{pmatrix}
			0\\ (c-\mathrm{i}dk) f_{n,k}(r)\mathrm{e}^{-\mathrm{i}k\theta}\\(c+\mathrm{i}dk)f_{n,k}(r)\mathrm{e}^{\mathrm{i}k\theta}
		\end{pmatrix}, c,d\in\mathbb{R}
	\end{aligned}\Big\},
\end{equation*}
which means 
	\begin{equation}\label{ran-bot-eq-ker}
	[\text{Ran} L(\beta_{0},\mu_0)]^{\bot}=\ker \big(L(\beta_{0},\mu_0)\big)^*=\ker L(\beta_{0},\mu_0).
	\end{equation}

\textbf{Step 2:}
%\textbf{Claim}:
 We claim that the complement of $\text{Ran} L(\beta_{0},\mu_0)$ is spanned by the vectors $L_{\beta}(\beta_{0},\mu_0)\textbf{h}_0$ and $L_{\mu}(\beta_{0},\mu_0)\textbf{h}_0$, where $\textbf{h}_0\in \ker L(\beta_{0},\mu_0)$ and 
\begin{equation*}
	\begin{aligned}
		\textbf{h}_0=	\begin{pmatrix}
			h_1\\h_2\\h_3
		\end{pmatrix}&=S\begin{pmatrix}
			0\\  f_{n,k}(r)\mathrm{e}^{-\mathrm{i}k\theta}\\f_{n,k}(r)\mathrm{e}^{\mathrm{i}k\theta}
		\end{pmatrix}=f_{n,k}(r)\begin{pmatrix}
		-\cos k\theta+\sqrt{3}\sin k\theta\\
		-\cos k\theta-\sqrt{3}\sin k\theta\\
		2\cos k\theta
		\end{pmatrix}.
	\end{aligned}
\end{equation*} Since \eqref{ran-bot-eq-ker}, it suffices to prove that $\ker L(\beta_{0},\mu_0)$ is spanned by $L_{\beta}(\beta_{0},\mu_0)\textbf{h}_0$ and $L_{\mu}(\beta_{0},\mu_0)\textbf{h}_0$.

By direct calculations, one has
\begin{equation*}
	\begin{aligned}
		 L_{\beta}(\beta_{0},\mu_0)\textbf{h}_0
		= t_{\beta_0,\mu_0}^2\begin{pmatrix}
			-(\alpha+\gamma)& \alpha & \gamma \\
			\gamma&
			-(\alpha+\gamma) &\alpha\\
			\alpha&\gamma&	-(\alpha+\gamma)
		\end{pmatrix}\begin{pmatrix}
			h_{1}\\h_{2}\\h_{3}
		\end{pmatrix},
	\end{aligned}
\end{equation*}
and
\begin{equation*}
	\begin{aligned}
		 L_{\mu}(\beta_{0},\mu_0)\textbf{h}_0
		= \frac{1}{(\mu_0+\beta_0\alpha+\beta_0\gamma)^2}\begin{pmatrix}
			\mu_0^2+2\beta_0\mu_0(\alpha+\gamma)& \beta_0^2\alpha(\alpha+\gamma)&\beta_0^2\gamma(\alpha+\gamma)\\
			\beta_0^2\gamma(\alpha+\gamma)&		\mu_0^2+2\beta_0\mu_0(\alpha+\gamma)& \beta_0^2\alpha(\alpha+\gamma)\\
			\beta_0^2\alpha(\alpha+\gamma)&\beta_0^2\gamma(\alpha+\gamma)&		\mu_0^2+2\beta_0\mu_0(\alpha+\gamma)
		\end{pmatrix}\begin{pmatrix}
			h_{1}\\h_{2}\\h_{3}
		\end{pmatrix}.
	\end{aligned}
\end{equation*}
Therefore it is sufficient to prove that there exists $\sigma, \tau\in\mathbb{R}$ such that
$(v_1,v_2,v_3)=\sigma  L_{\beta}(\beta_{0},\mu_0)\textbf{h}_0+\tau   L_{\mu}(\beta_{0},\mu_0)\textbf{h}_0$ for any $(v_1, v_2,v_3)\in \ker L(\beta_{0},\mu_0)$. Recalling the expression in \eqref{kernel-expression}, if $(v_1, v_2,v_3)\in \ker L(\beta_{0},\mu_0)$, then there exists $c, d\in\mathbb{R}$ such that 
\begin{equation*}
	\begin{aligned}
		\begin{pmatrix}
			v_1\\v_2\\v_3
		\end{pmatrix}&=S\begin{pmatrix}
			0\\ (c-\mathrm{i}dk) f_{n,k}(r)\mathrm{e}^{-\mathrm{i}k\theta}\\(c+\mathrm{i}dk)f_{n,k}(r)\mathrm{e}^{\mathrm{i}k\theta}
		\end{pmatrix}.
	\end{aligned}
\end{equation*}
Consequently, we only need to prove 
\begin{small}
	\begin{equation*}
		\begin{aligned}
			&S\begin{pmatrix}
				0\\ (c-\mathrm{i}dk) f_{n,k}(r)\mathrm{e}^{-\mathrm{i}k\theta}\\(c+\mathrm{i}dk)f_{n,k}(r)\mathrm{e}^{\mathrm{i}k\theta}
			\end{pmatrix}
			=\sigma t_{\beta_0,\mu_0}^2\begin{pmatrix}
			-(\alpha+\gamma)& \alpha & \gamma \\
			\gamma&
			-(\alpha+\gamma) &\alpha\\
			\alpha&\gamma&	-(\alpha+\gamma)
			\end{pmatrix}S\begin{pmatrix}
			0\\  f_{n,k}(r)\mathrm{e}^{-\mathrm{i}k\theta}\\f_{n,k}(r)\mathrm{e}^{\mathrm{i}k\theta}
			\end{pmatrix}
		\\
			&\quad+\tau  \frac{1}{(\mu_0+\beta_0\alpha+\beta_0\gamma)^2}\begin{pmatrix}
			\mu_0^2+2\beta_0\mu_0(\alpha+\gamma)& \beta_0^2\alpha(\alpha+\gamma)&\beta_0^2\gamma(\alpha+\gamma)\\
			\beta_0^2\gamma(\alpha+\gamma)&		\mu_0^2+2\beta_0\mu_0(\alpha+\gamma)& \beta_0^2\alpha(\alpha+\gamma)\\
			\beta_0^2\alpha(\alpha+\gamma)&\beta_0^2\gamma(\alpha+\gamma)&		\mu_0^2+2\beta_0\mu_0(\alpha+\gamma)
			\end{pmatrix}S\begin{pmatrix}
			0\\  f_{n,k}(r)\mathrm{e}^{-\mathrm{i}k\theta}\\f_{n,k}(r)\mathrm{e}^{\mathrm{i}k\theta}
			\end{pmatrix}.
		\end{aligned}
	\end{equation*}
\end{small}
Equivalently, by Proposition \ref{S}, it suffices to prove
\begin{small}
	\begin{equation*}
		\begin{aligned}
		&	\begin{pmatrix}
				0\\ (c-\mathrm{i}dk) f_{n,k}(r)\mathrm{e}^{-\mathrm{i}k\theta}\\(c+\mathrm{i}dk)f_{n,k}(r)\mathrm{e}^{\mathrm{i}k\theta}
			\end{pmatrix}
			=\sigma t_{\beta_0,\mu_0}^2S^{-1}\begin{pmatrix}
				-(\alpha+\gamma)& \alpha & \gamma \\
				\gamma&
				-(\alpha+\gamma) &\alpha\\
				\alpha&\gamma&	-(\alpha+\gamma)
			\end{pmatrix}S\begin{pmatrix}
				0\\  f_{n,k}(r)\mathrm{e}^{-\mathrm{i}k\theta}\\f_{n,k}(r)\mathrm{e}^{\mathrm{i}k\theta}
			\end{pmatrix}
			\\
			&\quad+\tau  \frac{1}{(\mu_0+\beta_0\alpha+\beta_0\gamma)^2}S^{-1}\begin{pmatrix}
				\mu_0^2+2\beta_0\mu_0(\alpha+\gamma)& \beta_0^2\alpha(\alpha+\gamma)&\beta_0^2\gamma(\alpha+\gamma)\\
				\beta_0^2\gamma(\alpha+\gamma)&		\mu_0^2+2\beta_0\mu_0(\alpha+\gamma)& \beta_0^2\alpha(\alpha+\gamma)\\
				\beta_0^2\alpha(\alpha+\gamma)&\beta_0^2\gamma(\alpha+\gamma)&		\mu_0^2+2\beta_0\mu_0(\alpha+\gamma)
			\end{pmatrix}S\begin{pmatrix}
				0\\  f_{n,k}(r)\mathrm{e}^{-\mathrm{i}k\theta}\\f_{n,k}(r)\mathrm{e}^{\mathrm{i}k\theta}
			\end{pmatrix}\\
			&=\sigma \begin{pmatrix}
				0&0 &0 \\
				0&
				p_{\beta_0,\mu_0} &0\\
				0&0&	\bar{p}_{\beta_0,\mu_0}
			\end{pmatrix}\begin{pmatrix}
				0\\  f_{n,k}(r)\mathrm{e}^{-\mathrm{i}k\theta}\\f_{n,k}(r)\mathrm{e}^{\mathrm{i}k\theta}
			\end{pmatrix}
			+\tau  \begin{pmatrix}
				1& 0&0\\
				0&		m_{\beta_0,\mu_0}& 0\\
				0&0&		\bar{m}_{\beta_0,\mu_0}
			\end{pmatrix}\begin{pmatrix}
				0\\  f_{n,k}(r)\mathrm{e}^{-\mathrm{i}k\theta}\\f_{n,k}(r)\mathrm{e}^{\mathrm{i}k\theta}
			\end{pmatrix}\\
			&=\begin{pmatrix}
				0\\ (\sigma p_{\beta_0,\mu_0}+\tau m_{\beta_0,\mu_0}) f_{n,k}(r)\mathrm{e}^{-\mathrm{i}k\theta}
				\\(\sigma \bar{p}_{\beta_0,\mu_0}+\tau \bar{m}_{\beta_0,\mu_0})f_{n,k}(r)\mathrm{e}^{\mathrm{i}k\theta}
			\end{pmatrix},
		\end{aligned}
	\end{equation*}
\end{small}
where \begin{equation*}
	\begin{aligned}
		p_{\beta_0,\mu_0}&=	\frac{1}{2(\mu_0+\beta_0\alpha+\beta_0\gamma)^2}	[-3\mu_0^2(\alpha+\gamma)+\mathrm{i}\sqrt{3}\mu_0^2(\alpha-\gamma)],\\
		m_{\beta_0,\mu_0}&=\frac{1}{2(\mu_0+\beta_0\alpha+\beta_0\gamma)^2}[2\mu_0^2+4\beta_0\mu_0(\alpha+\gamma)-\beta_0^2(\alpha+\gamma)^2+\mathrm{i}\sqrt{3}\beta_0^2(\alpha^2-\gamma^2)].
	\end{aligned}
\end{equation*} Thus
one only needs to prove that $ p_{\beta_0,\mu_0}$ and  $m_{\beta_0,\mu_0}$ are linearly independent. 
Actually, the determinant 
\begin{equation*}
	\begin{aligned}
		\begin{vmatrix}
			-3\mu_0^2(\alpha+\gamma)&\sqrt{3}\mu_0^2(\alpha-\gamma)\\
			2\mu_0^2+4\beta_0\mu_0(\alpha+\gamma)-\beta_0^2(\alpha+\gamma)^2&\sqrt{3}\beta_0^2(\alpha^2-\gamma^2)
		\end{vmatrix}			=-2\sqrt{3}\mu_0^2(\alpha-\gamma)(\mu_0+\beta_0\gamma+\beta_0\alpha)^2\neq 0.
	\end{aligned}
\end{equation*}
Hence the claim holds.

\textbf{Step 3:} By applying \cite[Theorem I.19.6]{bifurcation-2012},  the equation $\tilde{F}((u_1,u_2,u_3),(\beta,\mu))=\textbf{0}$ possesses a continuously differentiable solution curve $\big\{\big(\textbf{u}(s),(\beta(s),\mu(s))\big) | s\in(-\delta,\delta)\big\}\subset \big((C^{2,\alpha}(\overline{B_1},\mathbb{R}))^3\cap E\big)\times \mathbb{R}^2$ through $\big(\textbf{u}(0),(\beta(0),\mu(0))\big)=((0,0,0),(\beta_{0},\mu_0))$, $\textbf{u}(s)=s\textbf{h}_0+\psi\big(s\textbf{h}_0, (\beta(s),\mu(s)\big)$, and $\textbf{u}'(0)=\textbf{h}_0$.

\begin{remark}
By \cite[Remark I.19.7]{bifurcation-2012}, in fact, we could obtain a $2$-dimensional solution set 
	$$\Big\{\Big(s\textbf{h}+\psi\big(s\textbf{h},(\beta,\mu)(s,\textbf{h}))\big), (\beta,\mu)(s,\textbf{h})\Big)\Big\}$$
	 for $\tilde{F}(\textbf{u},(\beta,\mu))=0$ that depends in a continuously differentiable way on $(s,\textbf{h})$.
\end{remark}
%\medspace
\medskip
\subsection{Proof of Theorem \ref{Neumann}(iii)}
%	\subsubsection{Bifurcation-$(\beta,\omega)$}
We aim to utilize the multi-parameter bifurcation theorem, as described in \cite[Theorem I.19.6]{bifurcation-2012} or \cite[Theorem 1.3.4]{chang2005}  for fixed $\mu,\alpha,\gamma$.

Consider the mapping $G: \big((C^{2,\alpha}(\overline{B_1}))^3\cap E\big)\times \mathbb{R}^2\rightarrow (C^{0,\alpha}(\overline{B_1}))^3\cap E$:
\begin{equation*}
	G\big((u_1,u_2,u_3),(\beta,\omega)\big)=\begin{pmatrix}
		-\Delta u_1+\omega \textbf{x}^{\bot}\cdot\nabla u_1-\mu u_1(1-u_1)+\beta \alpha u_1 u_2+\beta \gamma u_1 u_3,\\
		-\Delta u_2+\omega \textbf{x}^{\bot}\cdot\nabla u_2-\mu u_2(1-u_2)+\beta \gamma u_1 u_2+\beta \alpha u_2 u_3\\
		-\Delta u_3+\omega \textbf{x}^{\bot}\cdot\nabla u_3-\mu u_3(1-u_3)+\beta \alpha u_1 u_3+\beta \gamma u_2 u_3
	\end{pmatrix}.
\end{equation*}
Obviously, 
\begin{equation*}
	G\big((1,1,1),(0,\omega)\big)=(0,0,0),\quad G\Big( t_{\beta}\big(1,1,1\big),(\beta,\omega)\Big)=(0,0,0),\ \text{where}\ t_{\beta}=\frac{\mu}{\mu+\beta(\alpha+\gamma)}.
\end{equation*}
Furthermore,
consider the mapping $\tilde{G}:  \big((C^{2,\alpha}(\overline{B_1}))^3\cap E\big)\times \mathbb{R}^2\rightarrow (C^{0,\alpha}(\overline{B_1}))^3\cap E$: 
\begin{equation*}
	\begin{aligned}
		\tilde{G}\big((u_1,u_2,u_3),(\beta,\omega)\big)&:={F}((u_1+t_{\beta},u_2+t_{\beta},u_3+t_{\beta}),(\beta,\omega))\\
		&=\begin{pmatrix}
			-\Delta u_1+\omega \textbf{x}^{\bot}\cdot\nabla u_1+(u_1+t_{\beta})(\mu u_1+\beta \alpha u_2+\beta \gamma u_3)\\
			-\Delta u_2+\omega \textbf{x}^{\bot}\cdot\nabla u_2+(u_2+t_{\beta})(\mu u_2+\beta \gamma u_1+\beta \alpha u_3)\\
			-\Delta u_3+\omega \textbf{x}^{\bot}\cdot\nabla u_3+(u_3+t_{\beta})(\mu u_3+\beta \alpha u_1+\beta \gamma u_2)
		\end{pmatrix}
		\\
		&=:L(\beta,\omega)\big(u_{1},u_{2},u_{3}\big)+P\big((u_{1},u_{2},u_{3}),(\beta,\omega)\big)
	\end{aligned}
\end{equation*}
where  
\begin{equation*}
	\begin{aligned}
		L(\beta,\omega)\big(u_{1},u_{2},u_{3}\big)&=\begin{pmatrix}
			-\Delta u_1+\omega \textbf{x}^{\bot}\cdot\nabla u_1+t_{\beta}(\mu u_1+\beta \alpha u_2+\beta \gamma u_3)\\
			-\Delta u_2+\omega \textbf{x}^{\bot}\cdot\nabla u_2+t_{\beta}(\mu u_2+\beta \gamma u_1+\beta \alpha u_3)\\
			-\Delta u_3+\omega \textbf{x}^{\bot}\cdot\nabla u_3+t_{\beta}(\mu u_3+\beta \alpha u_1+\beta \gamma u_2)
		\end{pmatrix},\\
		P\big((u_{1},u_{2},u_{3}),(\beta,\omega)\big)
		&=\begin{pmatrix}
			u_1(\mu u_1+\beta \alpha u_2+\beta \gamma u_3)\\
			u_2(\mu u_2+\beta \gamma u_1+\beta \alpha u_3)\\
			u_3(\mu u_3+\beta \alpha u_1+\beta \gamma u_2)
		\end{pmatrix}.
	\end{aligned}
\end{equation*}
Then 
$ \tilde{G}((0,0,0),(\beta,\omega))=0$, $P((0,0,0),(\beta,\omega))=0$ for all $\beta, \omega$, and  
\begin{equation*}
	\begin{aligned}
		P_x\big((u_1,u_2,u_3),(\beta,\omega)\big)(h_1,h_2,h_3)
		&
		=\begin{pmatrix}
			u_1(\mu h_1+\beta \alpha h_2+\beta \gamma h_3)+h_1(\mu u_1+\beta \alpha u_2+\beta \gamma u_3)\\
			u_2(\mu h_2+\beta \gamma h_1+\beta \alpha h_3)+h_2(\mu u_2+\beta \gamma u_1+\beta \alpha u_3)\\
			u_3(\mu h_3+\beta \alpha h_1+\beta \gamma h_2)+h_3(\mu u_3+\beta \alpha u_1+\beta \gamma u_2)
		\end{pmatrix},
	\end{aligned}
\end{equation*}
\begin{equation*}
	\begin{aligned}
		P_{x\beta}\big((u_1,u_2,u_3),(\beta,\omega)\big)(h_1,h_2,h_3)
		&
		=\begin{pmatrix}
			u_1(\alpha h_2+ \gamma h_3)+h_1( \alpha u_2+\gamma u_3)\\
			u_2( \gamma h_1+\alpha h_3)+h_2( \gamma u_1+\alpha u_3)\\
			u_3( \alpha h_1+ \gamma h_2)+h_3( \alpha u_1+ \gamma u_2)
		\end{pmatrix},
	\end{aligned}
\end{equation*}
and
\begin{equation*}
	\begin{aligned} 
		P_{x\omega}\big((u_1,u_2,u_3),(\beta,\omega)\big)(h_1,h_2,h_3)
		&
		=\begin{pmatrix}
			0\\
			0\\
			0
		\end{pmatrix}.
	\end{aligned}
\end{equation*}
Thus $P_x\big((0,0,0),(\beta,\omega)\big)=0$,
	$P_{x\beta}\big((0,0,0),(\beta,\omega)\big)=0$, and  $P_{x\omega}\big((0,0,0),(\beta,\omega)\big)=0$ for all $\beta, \omega$.

\textbf{Step 1:} We try to find the $\ker L(\beta,\omega)$ and $\text{Ran} L(\beta,\omega)$. Recall that
\begin{equation*}
	\begin{aligned}
		L(\beta,\omega)\big(h_{1},h_{2},h_{3}\big)	&
		= \begin{pmatrix}
			L_{\omega}(h_1)\\
			L_{\omega}(h_2)
			\\	L_{\omega}(h_3)
		\end{pmatrix}+	A_{\beta}
		\begin{pmatrix}
			h_1\\
			h_2\\
			h_3
		\end{pmatrix}, \text{where} \  A_{\beta}:=t_{\beta} \begin{pmatrix}
			\mu&\beta\alpha&\beta\gamma\\
			\beta\gamma&\mu&\beta\alpha\\
			\beta\alpha&\beta\gamma&\mu
		\end{pmatrix},
	\end{aligned}
\end{equation*}
For $A_{\beta}$, there are three eigenvalues: $\mu, \delta_{\beta}=\frac{2\mu-\beta(\alpha+\gamma)}{2}t_{\beta} +\mathrm{i}\frac{\sqrt{3}}{2}\beta(\alpha-\gamma)t_{\beta} $,  and $\bar{\delta}_{\beta}=\frac{2\mu-\beta(\alpha+\gamma)}{2}t_{\beta} -\mathrm{i}\frac{\sqrt{3}}{2}\beta(\alpha-\gamma)t_{\beta} $ by  Proposition \ref{S}, and
\begin{equation*}
	A_{\beta}S=S\begin{pmatrix}
		\mu&0&0\\
		0&\delta_{\beta}&0\\
		0&0&\bar{\delta}_{\beta}
	\end{pmatrix}, \quad S=	\begin{pmatrix}
		1&-\frac{1}{2}+\frac{\sqrt{3}}{2}\mathrm{i}&-\frac{1}{2}-\frac{\sqrt{3}}{2}\mathrm{i}\\
		1&-\frac{1}{2}-\frac{\sqrt{3}}{2}\mathrm{i}&-\frac{1}{2}+\frac{\sqrt{3}}{2}\mathrm{i}\\
		1&1&1
	\end{pmatrix}.
\end{equation*}
Then 
\begin{equation*}
	\begin{cases}
		L_{\omega} v_1+\mu v_1=0,\\
		L_{\omega} v_2+\delta_{\beta} v_2=0,\\
		L_{\omega} v_3+\bar{\delta}_{\beta} v_3=0,
	\end{cases} \text{where}	\begin{pmatrix}
		v_1\\v_2\\v_3
	\end{pmatrix}=S^{-1}\begin{pmatrix}
		h_1\\h_2\\h_3
	\end{pmatrix}.
\end{equation*}
By Proposition \ref{spectrum-countable}, it follows that   $\frac{\beta(\alpha+\gamma)-2\mu}{2}t_{\beta}$ is the eigenvalue of $-\Delta_N$, and $\frac{\sqrt{3}}{2}\beta(\alpha-\gamma)t_{\beta}=k\omega$ with $k\in \mathbb{Z}$, that is,
\begin{itemize}
	\item[(i)]  $\frac{\beta(\alpha+\gamma)-2\mu}{2}\frac{\mu}{\mu+\beta(\alpha+\gamma)}\in \{\lambda_{n}:  \lambda_{n} \text{ \ is\  a\ positive\  eigenvalue\  of\ } -\Delta_N, n=1,2,\cdots \}$;
	\item [(ii)] $\frac{\sqrt{3}}{2}\beta(\alpha-\gamma)\frac{\mu}{\mu+\beta(\alpha+\gamma)}\in \{k\omega: k\in\mathbb{Z}\}$. 
\end{itemize}
%Consequently,
%
Hence  for $\beta_{0}, \omega_0>0$ 
 satisfying  \eqref{beta-star2} with $\lambda_{n}$ and $k$  chosen properly and fixed,  
%Accordingly,
\begin{equation*}
	\ker L(\beta_{0},\omega_0)=\{(h_{1},h_{2},h_{3})\},\ \text{where}\ 
	\begin{pmatrix}
		h_1\\h_2\\h_3
	\end{pmatrix}=S\begin{pmatrix}
		0\\ (a-\mathrm{i}bk) f_{n,k}(r)\mathrm{e}^{-\mathrm{i}k\theta}\\(c+\mathrm{i}dk)f_{n,k}(r)\mathrm{e}^{\mathrm{i}k\theta}
	\end{pmatrix}.
\end{equation*}
To guarantee $h_{i}$ are real functions, we need $a=c, b=d$. Thus 
\begin{equation*}%\label{kernel-omega}
	\begin{aligned}
		\begin{pmatrix}
			h_1\\h_2\\h_3
		\end{pmatrix}&=S\begin{pmatrix}
			0\\ (a-\mathrm{i}bk) f_{n,k}(r)\mathrm{e}^{-\mathrm{i}k\theta}\\(a+\mathrm{i}bk)f_{n,k}(r)\mathrm{e}^{\mathrm{i}k\theta}
		\end{pmatrix}=af_{n,k}(r)\begin{pmatrix}
			-\cos k\theta+\sqrt{3}\sin k\theta\\
			-\cos k\theta-\sqrt{3}\sin k\theta\\
			2\cos k\theta
		\end{pmatrix}+bkf_{n,k}(r)\begin{pmatrix}
			\sin k\theta+\sqrt{3}\cos k\theta\\
			\sin k\theta-\sqrt{3}\cos k\theta\\
			-2\sin k\theta
		\end{pmatrix}.
	\end{aligned}
\end{equation*}
since $(h_{1},h_{2},h_{3})\in E$, it follows that \textcolor{black}{$\frac{2k+2}{3}\in\{2n: n\in\mathbb{Z}\}$.} 
In this case, 
$\dim \ker L(\beta_{0},\omega_0)= 2$. 
On the other hand, recall that
\begin{equation*}
	[\text{Ran} L(\beta_{0},\omega_0)]^{\bot}= \ker \big(L(\beta_{0},\omega_0)\big)^*.
\end{equation*}
In the following, we try to find $ \ker \big(L(\beta_{0},\omega_0)\big)^*$. Notice that
\begin{equation*}
	\begin{aligned}
		(L(\beta,\omega))^{*}(v_{1},v_{2},v_{3})
		= \begin{pmatrix}
			L_{-\omega}(v_1)\\
			L_{-\omega}(v_2)
			\\	L_{-\omega}(v_3)
		\end{pmatrix}+	M_{\beta}
		\begin{pmatrix}
			v_1\\
			v_2\\
			v_3
		\end{pmatrix}, \text{where} \  M_{\beta}:=t_{\beta} \begin{pmatrix}
			\mu&\beta\gamma&\beta\alpha\\
			\beta\alpha&\mu&\beta\gamma\\
			\beta\gamma&\beta\alpha&\mu
		\end{pmatrix}.
	\end{aligned}
\end{equation*}
Hence by Proposition \ref{S} again,
\begin{equation*}
	M_{\beta}S=S\begin{pmatrix}
		\mu&0&0\\
		0&\bar{\delta}_{\beta}&0\\
		0&0&{\delta}_{\beta}
	\end{pmatrix}, \quad S=	\begin{pmatrix}
		1&-\frac{1}{2}+\frac{\sqrt{3}}{2}\mathrm{i}&-\frac{1}{2}-\frac{\sqrt{3}}{2}\mathrm{i}\\
		1&-\frac{1}{2}-\frac{\sqrt{3}}{2}\mathrm{i}&-\frac{1}{2}+\frac{\sqrt{3}}{2}\mathrm{i}\\
		1&1&1
	\end{pmatrix}.
\end{equation*}
Then  for $(v_{1},v_{2},v_{3})\in \ker \big(L(\beta_{0},\omega_0)\big)^*$,
\begin{equation*}
	\begin{cases}
		L_{-\omega} w_1+\mu w_1=0,\\
		L_{-\omega} w_2+(-\lambda_{n}-\mathrm{i}\omega k) w_2=0,\\
		L_{-\omega} w_3+(-\lambda_{n}+\mathrm{i}\omega k) w_3=0,
	\end{cases} \text{where}	\begin{pmatrix}
		w_1\\w_2\\w_3
	\end{pmatrix}=S^{-1}\begin{pmatrix}
		v_1\\v_2\\v_3
	\end{pmatrix}.
\end{equation*}
Thus
\begin{equation*}
		\ker \big(L(\beta_{0},\omega_0)\big)^*=\Big\{(v_{1},v_{2},v_{3}),\ \text{where}\ 
	\begin{aligned}
		\begin{pmatrix}
			v_1\\v_2\\v_3
		\end{pmatrix}=S\begin{pmatrix}
			0\\ (c-\mathrm{i}dk) f_{n,k}(r)\mathrm{e}^{-\mathrm{i}k\theta}\\(c+\mathrm{i}dk)f_{n,k}(r)\mathrm{e}^{\mathrm{i}k\theta}
		\end{pmatrix}, c,d\in\mathbb{R}
	\end{aligned}\Big\},
\end{equation*}
which means \begin{equation}\label{range}
	 [\text{Ran} L(\beta_{0},\omega_0)]^{\bot}=\ker \big(L(\beta_{0},\omega_0)\big)^*=\ker L(\beta_{0},\omega_0)
\end{equation}

 \textbf{Step 2: }  We claim that the complement of $\text{Ran} L(\beta_{0},\omega_0)$ is spanned by the vectors $L_{\beta}(\beta_{0},\omega_0)\textbf{h}_0$ and $L'_{\omega}(\beta_{0},\omega_0)\textbf{h}_0$, where $\textbf{h}_0\in \ker L(\beta_{0},\omega_0)$ and 
\begin{equation*}
	\begin{aligned}
		\textbf{h}_0=	\begin{pmatrix}
			h_1\\h_2\\h_3
		\end{pmatrix}&=S\begin{pmatrix}
			0\\  f_{n,k}(r)\mathrm{e}^{-\mathrm{i}k\theta}\\f_{n,k}(r)\mathrm{e}^{\mathrm{i}k\theta}
		\end{pmatrix}=f_{n,k}(r)\begin{pmatrix}
		-\cos k\theta+\sqrt{3}\sin k\theta\\
		-\cos k\theta-\sqrt{3}\sin k\theta\\
		2\cos k\theta
		\end{pmatrix}.
	\end{aligned}
\end{equation*} Since \eqref{range}, it suffices to prove that $\ker L(\beta_{0},\omega_0)$ is spanned by $L_{\beta}(\beta_{0},\omega_0)\textbf{h}_0$ and $L'_{\omega}(\beta_{0},\omega_0)\textbf{h}_0$.
By straightforward calculations and Lemma \ref{polar-represent},
\begin{equation*}
	\begin{aligned}
	 L_{\beta}(\beta_{0},\omega_0)\textbf{h}_{0}= t_{\beta_0}^2\begin{pmatrix}
			-(\alpha+\gamma)& \alpha & \gamma \\
			\gamma&
			-(\alpha+\gamma) &\alpha\\
			\alpha&\gamma&	-(\alpha+\gamma)
		\end{pmatrix}\begin{pmatrix}
			h_{1}\\h_{2}\\h_{3}
		\end{pmatrix},
	\end{aligned}
\end{equation*}
and 
\begin{equation*}
	\begin{aligned}
	 L'_{\omega}(\beta_{0},\omega_0)\textbf{h}_{0}=
		 \begin{pmatrix}
			\textbf{x}^{\bot}\cdot \nabla h_{1}\\	\textbf{x}^{\bot}\cdot \nabla h_{2}\\	\textbf{x}^{\bot}\cdot \nabla h_{3}
		\end{pmatrix}=\begin{pmatrix}
	\partial_ \theta h_1\\
			\partial_ \theta h_2\\
				\partial_ \theta h_3
		\end{pmatrix}=S\begin{pmatrix}
		0\\  -\mathrm{i}k f_{n,k}(r)\mathrm{e}^{-\mathrm{i}k\theta}\\
		\mathrm{i}kf_{n,k}(r)\mathrm{e}^{\mathrm{i}k\theta}
		\end{pmatrix}.
	\end{aligned}
\end{equation*}
Therefore it is sufficient to prove that there exists $\sigma, \tau\in\mathbb{R}$ such that
$(v_1,v_2,v_3)=\sigma  L_{\beta}(\beta_{0},\omega_0)\textbf{h}_0+\tau   L'_{\omega}(\beta_{0},\mu_0)\textbf{h}_0$ for any $(v_1, v_2,v_3)\in \ker L(\beta_{0},\omega_0)$. Recalling the expression in \eqref{kernel-expression}, if $(v_1, v_2,v_3)\in \ker L(\beta_{0},\omega_0)$, then  there exists $c, d\in\mathbb{R}$ such that 
\begin{equation*}
	\begin{aligned}
		\begin{pmatrix}
			v_1\\v_2\\v_3
		\end{pmatrix}&=S\begin{pmatrix}
			0\\ (c-\mathrm{i}dk) f_{n,k}(r)\mathrm{e}^{-\mathrm{i}k\theta}\\(c+\mathrm{i}dk)f_{n,k}(r)\mathrm{e}^{\mathrm{i}k\theta}
		\end{pmatrix}.
	\end{aligned}
\end{equation*}
Consequently, we only need to prove 
\begin{small}
	\begin{equation*}
		\begin{aligned}
			S\begin{pmatrix}
				0\\ (c-\mathrm{i}dk) f_{n,k}(r)\mathrm{e}^{-\mathrm{i}k\theta}\\(c+\mathrm{i}dk)f_{n,k}(r)\mathrm{e}^{\mathrm{i}k\theta}
			\end{pmatrix}&=\sigma t_{\beta_0,\mu_0}^2\begin{pmatrix}
				-(\alpha+\gamma)& \alpha & \gamma \\
				\gamma&
				-(\alpha+\gamma) &\alpha\\
				\alpha&\gamma&	-(\alpha+\gamma)
			\end{pmatrix}S\begin{pmatrix}
				0\\  f_{n,k}(r)\mathrm{e}^{-\mathrm{i}k\theta}\\f_{n,k}(r)\mathrm{e}^{\mathrm{i}k\theta}
			\end{pmatrix}
		+\tau  S\begin{pmatrix}
				0\\  -\mathrm{i}k f_{n,k}(r)\mathrm{e}^{-\mathrm{i}k\theta}\\
				\mathrm{i}kf_{n,k}(r)\mathrm{e}^{\mathrm{i}k\theta}
			\end{pmatrix}.
		\end{aligned}
	\end{equation*}
\end{small}
Equivalently, by Proposition \ref{S}, it suffices to prove
\begin{small}
	\begin{equation*}
		\begin{aligned}
			\begin{pmatrix}
				0\\ (c-\mathrm{i}dk) f_{n,k}(r)\mathrm{e}^{-\mathrm{i}k\theta}\\(c+\mathrm{i}dk)f_{n,k}(r)\mathrm{e}^{\mathrm{i}k\theta}
			\end{pmatrix}&=\sigma t_{\beta_0,\mu_0}^2S^{-1}\begin{pmatrix}
				-(\alpha+\gamma)& \alpha & \gamma \\
				\gamma&
				-(\alpha+\gamma) &\alpha\\
				\alpha&\gamma&	-(\alpha+\gamma)
			\end{pmatrix}S\begin{pmatrix}
				0\\  f_{n,k}(r)\mathrm{e}^{-\mathrm{i}k\theta}\\f_{n,k}(r)\mathrm{e}^{\mathrm{i}k\theta}
			\end{pmatrix}+
		\tau  \begin{pmatrix}
			0\\  -\mathrm{i}k f_{n,k}(r)\mathrm{e}^{-\mathrm{i}k\theta}\\
			\mathrm{i}kf_{n,k}(r)\mathrm{e}^{\mathrm{i}k\theta}
		\end{pmatrix}\\
			&=\sigma \begin{pmatrix}
				0&0 &0 \\
				0&
				p_{\beta_0,\mu_0} &0\\
				0&0&	\bar{p}_{\beta_0,\mu_0}
			\end{pmatrix}\begin{pmatrix}
				0\\  f_{n,k}(r)\mathrm{e}^{-\mathrm{i}k\theta}\\f_{n,k}(r)\mathrm{e}^{\mathrm{i}k\theta}
			\end{pmatrix}
		+
		\tau  \begin{pmatrix}
			0\\  -\mathrm{i}k f_{n,k}(r)\mathrm{e}^{-\mathrm{i}k\theta}\\
			\mathrm{i}kf_{n,k}(r)\mathrm{e}^{\mathrm{i}k\theta}
		\end{pmatrix}\\
			&=\begin{pmatrix}
				0\\ (\sigma p_{\beta_0,\mu_0}-	\mathrm{i}k\tau ) f_{n,k}(r)\mathrm{e}^{-\mathrm{i}k\theta}
				\\(\sigma \bar{p}_{\beta_0,\mu_0}+	\mathrm{i}k\tau) )f_{n,k}(r)\mathrm{e}^{\mathrm{i}k\theta}
			\end{pmatrix},
		\end{aligned}
	\end{equation*}
\end{small}
where \begin{equation*}
	\begin{aligned}
		p_{\beta_0,\mu_0}&=	\frac{\mu_0^2}{2(\mu_0+\beta_0\alpha+\beta_0\gamma)^2}	[-3(\alpha+\gamma)+\mathrm{i}\sqrt{3}(\alpha-\gamma)].
	\end{aligned}
\end{equation*} Thus
one only needs to prove that $ p_{\beta_0,\mu_0}$ and  $\mathrm{i}k$ are linearly independent, which is obvious.
Hence the claim holds.

\textbf{Step 3:} By applying \cite[Theorem I.19.6]{bifurcation-2012},  the equation $\tilde{G}((u_1,u_2,u_3),(\beta,\omega))=0$ possesses a continuously differentiable solution curve $\big\{\big(\textbf{u}(s),(\beta(s),\omega(s))\big) | s\in(-\delta,\delta)\big\}\subset \big((C^{2,\alpha}(\overline{B_1},\mathbb{R}))^3\cap E\big)\times \mathbb{R}^2$ through $\big(\textbf{u}(0),(\beta(0),\omega(0))\big)=((0,0,0),(\beta_{0},\omega_0))$, $\textbf{u}(s)=s\textbf{h}_0+\psi\big(s\textbf{h}_0, (\beta(s),\omega(s)\big)$, and $\textbf{u}'(0)=\textbf{h}_0$.
\begin{remark}
	By \cite[Remark I.19.7]{bifurcation-2012}, 
as a matter of fact, we  obtain a $2$-dimensional solution set 
$$\Big\{\Big(s\textbf{h}+\psi\big(s\textbf{h},(\beta,\omega)(s,\textbf{h}))\big), (\beta,\omega)(s,\textbf{h})\Big)\Big\}$$ for $\tilde{G}(\textbf{u},(\beta,\omega))=0$ that depends in a continuously differentiable way on $(s,\textbf{h})$ and $\textbf{u}'(0)=\textbf{h}$.
\end{remark}

\subsection{Proof of Theorem \ref{Neumann}(iv)}
The statement holds directly from the above analysis in Sections 2.2,  2.3.

\subsection{Instability}
	The stability of a positive solution $(u_1,u_2, u_3)$ to \eqref{sys-ome} is determined by the eigenvalue problem:
\begin{equation}\label{eigenvalue-stability}
	\begin{cases}
		L_{\omega}\xi=\lambda \xi+(\mu-2\mu u_1-\beta\alpha u_2-\beta\gamma u_3)\xi-\beta\alpha u_1 \eta-\beta \gamma u_1\zeta,& \text{in}\ B_1,\\
		L_{\omega}\eta=\lambda \eta+(\mu-2\mu u_2-\beta\alpha u_3-\beta\gamma u_1)\eta-\beta \gamma u_2\xi-\beta\alpha u_2 \zeta,& \text{in}\ B_1,\\
		L_{\omega}\zeta=\lambda \zeta+(\mu-2\mu u_3-\beta\alpha u_1-\beta\gamma u_2)\zeta-\beta\alpha u_3 \xi-\beta \gamma u_3\eta,& \text{in}\ B_1,
		%\\
		%\frac{\partial \xi}{\partial n}=	\frac{\partial \eta}{\partial n}=	\frac{\partial \zeta}{\partial n}=0,&\text{on}\ \partial B_1.
	\end{cases}
\end{equation}
with Neumann or Dirichlet boundary conditions. 
We consider the principal eigenvalue $\lambda(u_1,u_2,u_3)$, which has the smallest real part among all the spectrum points. Then the  solution $(u_1,u_2,u_3)$ is unstable if $\Re \lambda(u_1,u_2,u_3)<0$.
\begin{proof}[Proof of Theorem \ref{Neumann}(v)]
Let  $(u_1,u_2,u_3)=t_{\beta}(1,1,1)$ in \eqref{eigenvalue-stability} with $t_\beta=\frac{\mu}{\mu+\beta\alpha+\beta\gamma}$,  then it is reduced to 
\begin{equation*}%\label{eigenvalue-stability-c}
	\begin{cases}
		L_{\omega}\xi=\lambda\xi-t_\beta (\mu  \xi+\beta\alpha  \eta+\beta \gamma \zeta),& \text{in}\ B_1,\\
		L_{\omega}\eta=\lambda\eta-t_\beta(\mu  \eta+\beta \gamma \xi+\beta\alpha \zeta),& \text{in}\ B_1,\\
		L_{\omega}\zeta=\lambda\zeta-t_\beta(\mu  \zeta+\beta\alpha  \xi+\beta \gamma\eta),& \text{in}\ B_1,\\
		\frac{\partial \xi}{\partial n}=	\frac{\partial \eta}{\partial n}=	\frac{\partial \zeta}{\partial n}=0,&\text{on}\ \partial B_1.
	\end{cases}
\end{equation*}
Equivalently, 
\begin{equation*}
	\begin{aligned}
		%	L(\beta,\omega)\big(h_{1},h_{2},h_{3}\big)	&
		%	=
		\begin{pmatrix}
			(L_{\omega}-\lambda)\xi\\
			(L_{\omega}-\lambda)\eta
			\\	(L_{\omega}-\lambda)\zeta
		\end{pmatrix}+	A_{\beta}
		\begin{pmatrix}
			\xi\\
			\eta\\
			\zeta
		\end{pmatrix}=	\begin{pmatrix}
			0\\
			0\\
			0
		\end{pmatrix}, \text{where} \  A_{\beta}:=t_{\beta} \begin{pmatrix}
			\mu&\beta\alpha&\beta\gamma\\
			\beta\gamma&\mu&\beta\alpha\\
			\beta\alpha&\beta\gamma&\mu
		\end{pmatrix}.
	\end{aligned}
\end{equation*}
For $A_{\beta}$, there are three eigenvalues: $\mu, \delta_{\beta}=\frac{2\mu-\beta(\alpha+\gamma)}{2}t_{\beta} +\mathrm{i}\frac{\sqrt{3}}{2}\beta(\alpha-\gamma)t_{\beta} $,  and $\bar{\delta}_{\beta}=\frac{2\mu-\beta(\alpha+\gamma)}{2}t_{\beta} -\mathrm{i}\frac{\sqrt{3}}{2}\beta(\alpha-\gamma)t_{\beta} $ by  Proposition \ref{S}, and
\begin{equation*}
	\begin{cases}
		(L_{\omega}-\lambda) v_1+\mu v_1=0,\\
		(L_{\omega}-\lambda) v_2+\delta_{\beta} v_2=0,\\
		(L_{\omega}-\lambda) v_3+\bar{\delta}_{\beta} v_3=0,
	\end{cases} \text{where}	\begin{pmatrix}
		v_1\\v_2\\v_3
	\end{pmatrix}=S^{-1}\begin{pmatrix}
		\xi \\ \eta\\ \zeta
	\end{pmatrix}.
\end{equation*}
For $\beta>\frac{2\mu}{\alpha+\gamma}$, then $\min_{\lambda} \Re \lambda= \frac{2\mu-\beta(\alpha+\gamma)}{2}\frac{\mu}{\mu+\beta\alpha+\beta\gamma}<0$. Therefore the solution $\frac{\mu}{\mu+\beta\alpha+\beta\gamma}(1,1,1)$ is unstable if $\beta>\frac{2\mu}{\alpha+\gamma}$.
\end{proof}

\section{The DBC system}
We consider the nonlinear elliptic system
\begin{equation}\label{system2-dp}
	\begin{cases}
		L_{\omega}u_1=\mu u_1(1-u_1)-\beta \alpha u_1u_2-\beta \gamma u_1 u_3,& \text{in}\ B_1,\\
		L_{\omega}u_2=\mu u_2(1-u_2)-\beta \gamma u_1u_2-\beta \alpha u_2 u_3,& \text{in}\ B_1,\\
		L_{\omega}u_3=\mu u_3(1-u_3)-\beta \alpha u_1u_3-\beta \gamma u_2 u_3,& \text{in}\ B_1,\\
	(u_1,u_2,u_3)=(\phi_{1},\phi_{2},\phi_{3}),&\text{on}\ \partial B_1,
	\end{cases}
\end{equation}
with $0\le \phi_{i}(\textbf{x})< 1$, $\phi_{i}\not\equiv 0$ and \begin{equation*}
\phi_{2}(\textbf{x})=\phi_{1}(R_{\frac{2\pi}{3}}\textbf{x}), \ \phi_{3}(\textbf{x})=\phi_{2}(R_{\frac{2\pi}{3}}\textbf{x})\  \text{on}\ \partial B_1.
\end{equation*}

We want to prove the existence of nontrivial positive ($	u_i>0 \ \text{in}\  B_{1},\  \forall i=1,2,3$) symmetric solutions such that
\begin{equation*}
	u_2(\textbf{x})=u_1(R_{\frac{2\pi}{3}}\textbf{x}), \  u_3(\textbf{x})=u_2(R_{\frac{2\pi}{3}}\textbf{x}).
\end{equation*}
\begin{lemma}%\label{upperbound-ui-dp}
	Any positive solution of \eqref{system2-dp} is  such that $0<u_i<1$ in ${B_1}$ for any $i=1, 2, 3$.
\end{lemma}
\begin{proof}
Take $v_i=u_i-1$, then     
	\begin{equation*}
		\begin{cases}
			(L_{\omega}+\mu) v_i\le 0,& \text{in}\ B_1,\\
			v_i<  0,& \text{on}\  \partial B_1.
		\end{cases}
	\end{equation*}
	By \cite[Corollary 3.2]{trudinger},  $u_i\le 1$  in $B_1$. Furthermore,
	if $u_i(\textbf{x}_0)=\max_{B_1} u_i=1$, then $-\Delta u_i(\textbf{x}_0)\ge 0$ and $\nabla u_i(\textbf{x}_0)=0$, 
	hence 
	$
	0\le L_{\omega} u_i(\textbf{x}_0)<0,
	$
	a contradiction.
	We conclude that $0<u_i<1$  in ${B_1}$.
\end{proof}

Consider the space 
\begin{equation*}%\label{E-sp-dp}
	\tilde{E}:=\{(u_1, u_2, u_3): u_{i}\in C(\overline{B_{1}}), u_i(\textbf{x})=\phi_{i}(\textbf{x}) \ \text{on}\ \partial B_1,  u_2(\textbf{x})=u_1(R_{\frac{2\pi}{3}}\textbf{x}), u_3(\textbf{x})=u_2(R_{\frac{2\pi}{3}}\textbf{x}) \},
\end{equation*}
with the norm 
\begin{equation*}
	\|(u_1, u_2, u_3)\|_{\tilde{E}}:=\max_{\textbf{x}\in\overline{B_1}}|u_1(\textbf{x})|=\max_{\textbf{x}\in\overline{B_1}}|u_2(\textbf{x})|=\max_{\textbf{x}\in\overline{B_1}}|u_3(\textbf{x})|.
\end{equation*}

\begin{proof}[Proof of Theorem \ref{Dirichlet}]
		For the system 
	\begin{equation}\label{system-inverse2nn-dp}
		\begin{cases}
			(L_{\omega}+\mu+2\beta\alpha+2\beta\gamma) z_1=\mu u_1(2-u_1)+2\beta(\alpha+\gamma) u_{1}-\beta \alpha u_1u_2-\beta \gamma u_1 u_3,& \text{in}\ B_1,\\
			(L_{\omega}+\mu+2\beta\alpha+2\beta\gamma) z_2=\mu u_2(2-u_2)+2\beta(\alpha+\gamma)  u_{2}-\beta \gamma u_1u_2-\beta \alpha u_2 u_3,& \text{in}\ B_1,\\
			(L_{\omega}+\mu+2\beta\alpha+2\beta\gamma)z_3=\mu u_3(2-u_3)+2\beta(\alpha+\gamma)  u_{3}-\beta \alpha u_1u_3-\beta \gamma u_2 u_3,& \text{in}\ B_1,\\
			(z_1,z_2,z_3)=(\phi_{1},\phi_{2},\phi_{3}),&\text{on}\ \partial B_1,
		\end{cases}
	\end{equation}
	we try to find positive fixed points in $\tilde{E}$.  Define a map
	\begin{equation*}
		\Phi(u_{1}, u_{2},u_{3})=(z_{1}, z_{2}, z_{3}) \Leftrightarrow (z_{1}, z_{2}, z_{3}) \ \text{solves}\  \eqref{system-inverse2nn-dp}.
	\end{equation*}
	Suppose 
	\begin{equation*}
		(u_{1}, u_{2},u_{3})\in B:=\{(u_{1}, u_{2},u_{3})\in E: 0\le u_{i}(\textbf{x})\le 2, i=1,2,3\}.
	\end{equation*}
	Firstly, it is easy to see that the RHS of each equation must be non-negative and less than $\mu+4\beta(\alpha+\gamma)$. By \cite[Theorem 3.5]{trudinger}, we know $z_{i}> 0$ in ${B_{1}}$. 
	
\textbf{Claim.} $\Phi: B\rightarrow B$ is a compact operator.
	 Set $v_{i}=z_{i}-2$, then 
	\begin{equation*}
		\begin{cases}
			(L_{\omega}+\mu+2\beta\alpha+2\beta\gamma)v_{i}<0, & \text{in} \ B_{1},\\
			v_{i}<0,&\text{on} \ \partial B_{1}.
		\end{cases}
	\end{equation*}
	Hence $0\le z_{i}\le 2$ by the maximum principle.
	Thus the claim holds by \cite[Theorem 4.2]{amann-1976} or \cite[Theorem 2.4.2.5]{neumann-book}. 
	
	Notice that $B$ is a bounded convex closed set in $\tilde{E}$. If $\Phi\big((u_{1}, u_{2}, u_{3})\big)\neq (u_{1}, u_{2}, u_{3})$ ( $\forall (u_{1}, u_{2}, u_{3})\in\partial B$), then by the Rothe fixed point theorem \cite[Corollary 13.1]{ziedler1993},
	\begin{equation*}
		\text{deg}(I-\Phi, B, (0,0,0))=1.
	\end{equation*}
	Thus there exists a fixed point, which means we have a positive solution for \eqref{system2-dp}.
\end{proof}

\section{Appendix}

\medskip
\subsection{Determinant of the matrix}
Consider the matrix \begin{equation*}
	M(\beta)=	\begin{pmatrix}
		\mu&\beta\alpha&\beta\gamma\\
		\beta\gamma&\mu&\beta\alpha\\
		\beta\alpha&\beta\gamma&\mu
	\end{pmatrix}.
\end{equation*}
The determinant of $M(\beta)$ is $(\alpha^3+\gamma^3)\beta^3-3\alpha\gamma\mu \beta^2+\mu^3$, which is always positive  for $\beta\in(0+\infty)$ by the following proposition.
\begin{prop}\label{det-positive}
	Consider the function $f(x)=(\alpha^3+\gamma^3)x^3-3\alpha\gamma\mu x^2+\mu^3$ with $\alpha>\gamma>0$, then $f(x)>0$ in $(0,+\infty)$.
\end{prop}
\begin{proof}
	By direct calculation,
	$f'(x)=3x\big[(\alpha^3+\gamma^3)x-2\alpha\gamma\mu\big]$, thus 
	\begin{equation*}
		\begin{aligned}
			\min_{x\in(0+\infty)} f(x)=f\big(\frac{2\alpha\gamma\mu}{\alpha^3+\gamma^3}\big)&=\frac{8\alpha^3\gamma^3\mu^3}{(\alpha^3+\gamma^3)^2}-3\alpha\gamma\mu\frac{4\alpha^2\gamma^2\mu^2}{(\alpha^3+\gamma^3)^2}+\mu^3\\
			&=\frac{\mu^3}{(\alpha^3+\gamma^3)^2}\big[-4\alpha^3\gamma^3+(\alpha^3+\gamma^3)^2\big]\\
			&=\frac{\mu^3}{(\alpha^3+\gamma^3)^2}(\alpha^3-\gamma^3)^2>0.
		\end{aligned}
	\end{equation*}
\end{proof}
Actually $(\alpha^3+\gamma^3)\beta^3-3\alpha\gamma\mu \beta^2+\mu^3=(\mu+\beta\alpha+\beta\gamma)\big[\beta^2(\alpha^2+\gamma^2-\alpha\gamma)-\beta\mu(\alpha+\gamma)+\mu^2\big]$. Actually the only real root of $f(x)$ is $x=-\frac{\mu}{\alpha+\gamma}$. 
Moreover, the inverse of $M(\beta)$ is 
\begin{equation*}
	M(\beta)^{-1}=\frac{1}{(\alpha^3+\gamma^3)\beta^3-3\alpha\gamma\mu \beta^2+\mu^3}	\begin{pmatrix}
		\mu^2-\alpha\gamma\beta^2&\gamma^2\beta^2-\alpha\mu \beta& \alpha^2\beta^2-\gamma\mu\beta\\
		\alpha^2\beta^2-\gamma\mu\beta&	\mu^2-\alpha\gamma\beta^2&\gamma^2\beta^2-\alpha\mu \beta\\
		\gamma^2\beta^2-\alpha\mu \beta&	\alpha^2\beta^2-\gamma\mu\beta&	\mu^2-\alpha\gamma\beta^2
	\end{pmatrix}.
\end{equation*}

\medskip
\subsection{Diagonalization of the matrix}
\begin{prop}\label{S}
	Assume $\tau>\delta$, then
	there exists an invertible matrix $S$ such that 
	\begin{equation*}
		AS=S\begin{pmatrix}
			\mu+\tau+\delta&0&0\\0&m&0\\0&0&\bar{m}
		\end{pmatrix}, \ \text{where}\  
		A=\begin{pmatrix}
			\mu&\tau&\delta\\
			\delta&\mu&\tau\\
			\tau&\delta&{\mu}
		\end{pmatrix},
		S=	\begin{pmatrix}
			1&-\frac{1}{2}+\frac{\sqrt{3}}{2}\mathrm{i}&-\frac{1}{2}-\frac{\sqrt{3}}{2}\mathrm{i}\\
			1&-\frac{1}{2}-\frac{\sqrt{3}}{2}\mathrm{i}&-\frac{1}{2}+\frac{\sqrt{3}}{2}\mathrm{i}\\
			1&1&1
		\end{pmatrix},
	\end{equation*}
	where $m=\frac{2\mu-\tau-\delta}{2}+\mathrm{i}\frac{\sqrt{3}(\tau-\delta)}{2}$. Moreover,
	\begin{equation*}
		S^{-1}=	\frac16\begin{pmatrix}
			2&2&2\\
			-1-\sqrt{3}\mathrm{i}&-1+\sqrt{3}\mathrm{i}&2\\
			-1+\sqrt{3}\mathrm{i}&-1-\sqrt{3}\mathrm{i}&2
		\end{pmatrix}.
	\end{equation*}
\end{prop}
\begin{proof}
	The conclusion is true by direct calculations.
\end{proof}

\end{document}